\documentclass[11pt,reqno]{amsart}

\usepackage{a4wide}
\usepackage{hyperref,amsmath,amsfonts,mathscinet,amssymb,amsthm}
\usepackage{accents}
\usepackage{enumitem}
\usepackage{constants}
\usepackage{xcolor}
\usepackage{tikz-cd}

\setlist[enumerate]{label={\rm(\roman*)}}

\theoremstyle{plain}
\newtheorem{theorem}{Theorem}[section]

\newtheorem{lemma}[theorem]{Lemma}

\theoremstyle{definition}
\newtheorem{definition}[theorem]{Definition}
\theoremstyle{remark}

\numberwithin{equation}{section}

\expandafter\let\expandafter\oldproof\csname\string\proof\endcsname
\let\oldendproof\endproof
\renewenvironment{proof}[1][\proofname]{%
  \oldproof[\bf #1]%
}{\oldendproof}

\def\M{\mathcal M}
\def\MM{\mathcal M_+}
\def\N{\mathbb N}
\def\Z{\mathbb Z}
\def\R{\mathbb R}

\def\esssup{\operatornamewithlimits{ess\sup}}

\newcommand{\ds}{\mathrm{\,d}s}
\newcommand{\dt}{\mathrm{\,d}t}
\newcommand{\dx}{\mathrm{\,d}x}
\newcommand{\dy}{\mathrm{\,d}y}

\newcommand{\jq}{\frac1q}
\newcommand{\jp}{\frac1p}

\newcommand{\mjp}{{-\frac1p}}

\newcommand{\qp}{{\frac qp}}

\newcommand{\sqa}{\{a_n\}}
\newcommand{\sqb}{\{b_n\}}
\newcommand{\sqc}{\{c_n\}}
\newcommand{\sqv}{\{v_n\}}
\newcommand{\sqw}{\{w_n\}}
\newcommand{\squ}{\{u_n\}_{n\in\Z}}

\newcommand{\RpZ}{\mathbb{R}_+^\mathbb{Z}}

\newcommand{\sumZ}{\sum_{n\in\Z}}

\def\supp{\operatorname{supp}}

\allowdisplaybreaks

\begin{document}

\title{Weighted inequalities for discrete iterated kernel operators}

\author{Amiran Gogatishvili, Lubo\v s Pick and Tu\u{g}\c{c}e \"{U}nver}

\email[A.~Gogatishvili]{gogatish@math.cas.cz}
\urladdr{0000-0003-3459-0355}
\email[L.~Pick]{pick@karlin.mff.cuni.cz}
\urladdr{0000-0002-3584-1454}
\email[T.~\"{U}nver]{tugceunver@kku.edu.tr}
\urladdr{0000-0003-0414-8400}

\address{
 Institute of Mathematics,
 of the Czech Academy of Sciences,
 \v Zitn\'a~25,
 115~67 Praha~1,
 Czech Republic}

\address{
 Department of Mathematical Analysis,
	Faculty of Mathematics and Physics,
	Charles University,
	Sokolovsk\'a~83,
	186~75 Praha~8,
	Czech Republic}

\address{
Department of Mathematics,
Faculty of Science and Arts,
 Kirikkale University,
  71450 Yahsihan,
    Kirikkale, Turkey}
	
\subjclass[2000]{46E30, 26D20, 47B38, 46B70}
\keywords{Weighted inequality, kernel operator, supremum operator}

\thanks{This research was supported by the grant P201-18-00580S of the Czech Science Foundation. The research of A. Gogatishvili was partially supported by Shota Rustaveli National Science Foundation (SRNSF), grant no: FR17-589. }

\begin{abstract}

We develop a new method that enables us to solve the open problem of characterizing discrete inequalities for kernel operators involving suprema. More precisely, we establish necessary and sufficient conditions under which
there exists a positive constant $C$ such that
\begin{equation*}
	\Bigg(\sum_{n\in\Z}\Bigg(\sum_{i=-\infty}^n{U}(i,n)a_i\Bigg)^{q} {w}_n\Bigg)^{\frac 1q}
	\le C
	\Bigg(\sum_{n\in\Z}a_n\sp p{v}_n\Bigg)\sp{\frac 1p}
\end{equation*}
holds for every sequence of nonnegative numbers $\{a_n\}_{n\in\Z}$
where $U$ is a kernel satisfying certain regularity condition, $0 < p,q \leq \infty$ and $\{v_n\}_{n\in\Z}$ and $\{w_n\}_{n\in\Z}$ are fixed weight sequences. We do the same for the inequality
\begin{equation*}
	\Bigg( \sumZ w_n \Big[ \sup_{-\infty<i\le n} U(i,n) \sum_{j=-\infty}^{i} a_j \Big]^q \Bigg)^\jq
	\le C
	\Bigg( \sumZ a_n^p v_n \Bigg)^\jp.
\end{equation*}
We characterize these inequalities by conditions of both discrete and continuous nature.
\end{abstract}


\maketitle

	\section{Introduction}

In this paper we focus on boundedness properties of \textit{supremum operators} involving kernels. We study both their discrete and continuous versions and the relations between them. The bridge between the continuous and the discrete world turns out to be very useful for solving difficult questions in either of them, as we shall demonstrate.

While classical Hardy-type operators involving sums and integrals have been extensively studied for over a century now, supremum operators constitute a relatively new object as their importance has been spotted not so long ago and they have been put under rather
a serious scrutiny over approximately last two decades.

The first significant appearance of a supremum operator materialized in~\cite{CKOP} where it was shown that the Calder\'on operator of the fractional maximal operator takes the form of a supremum operator. More precisely, it was shown that if $n\in\N$, $0 < \gamma < n$ and $M_{\gamma}$ is the fractional maximal operator defined by
\begin{equation*}
	M_{\gamma}f(x)=\sup_{Q\owns x}\frac{1}{|Q|^{1-\frac{\gamma}{n}}} \int_Q|f(y)|\,dy\quad \text{for every $f\in L^{1}_{\operatorname{loc}}(\mathbb R^n)$ and $x\in \mathbb R^n$},
\end{equation*}
(where the supremum is extended over all cubes $Q\subset\mathbb R^n$ with sides parallel to the
coordinate axes and $|E|$ denotes the $n$-dimensional Lebesgue measure of $E\subset\mathbb  R^n$),
then there exists a constant $C$ depending only on $n$ and $\gamma$ such that
\begin{equation*}
	(M_{\gamma}f)^{*}(t)\leq C \sup_{t < s < \infty} s^{\frac{\gamma}{n}-1}\int_0\sp{t}f^{*}(\tau)\,d\tau \quad \text{for every $f\in L^{1}_{\operatorname{loc}}(\mathbb R^n)$ and $t\in(0,\infty)$},
\end{equation*}
where $f\mapsto f^*$ denotes the operation of taking the nonincreasing rearrangement of a~function. Moreover, this pointwise estimate is sharp in the sense that there is a constant $c$, again depending only on $n$ and $\gamma$, such that for every decreasing function $\varphi$ on $(0,\infty)$ one can find a function $f$ on $\mathbb R^n$ such that $f^*=\varphi$ almost everywhere on $(0,\infty)$ and
\begin{equation*}
	(M_{\gamma}f)^{*}(t)\ge c \sup_{t < s < \infty} s^{\frac{\gamma}{n}-1}\int_0\sp{t}\varphi(\tau)\,d\tau \quad \text{for every $t\in(0,\infty)$}.
\end{equation*}
The estimate holds also when $\gamma=0$, yielding thereby a~significant extension of the classical Wiener--Riesz--Herz inequality.

An earlier occurrence of operation of taking a supremum can be found for example in~\cite{CwiPu1998}, see also~\cite{Pus1999}, where a different type of an operator taking suprema is studied in certain duality relationships.

What started as a~rudimentary observation has been later extended to a worthwhile theory. The next significant step was taken in~\cite{GOP}, where general weighted inequalities for supremum operators were characterized for the first time. In the meantime, supremum operators have been showing up unexpectedly in many different projects. In particular, they play a key role in the characterization of optimal Sobolev embeddings (see e.g.~\cite{T2,T3,CPS,len15,T4,CP-TRANS} and many more). Some of the results in~\cite{GOP} were not entirely satisfactory as the characterizing conditions for some of the inequalities were expressed in terms of discretizing sequences which made their verification virtually impossible in practice. This obstacle was to a~large extent overcome with the development of anti-discretization techniques, which took place in~\cite{GP-discr}. In that paper alone several open problems were solved (for example, previously unknown integral characterizations of embeddings between certain types of classical Lorentz spaces were nailed down). Further important techniques were discovered for instance in~\cite{Sinn-2003,Sinn03} and~\cite{CGMP2}. Significant applications were found for example in the regularity theory for partial differential equations (see e.g.~\cite{AlbCia2007,AlbCia2007a}), for improvement of sharp Sobolev inequalities (\cite{CiaFer2012}), in the theory of Besov spaces (\cite{CaeGO2008,Gol2007}), in a~general theory of function spaces (\cite{Cie2009}), in interpolation theory (\cite{CobGOP2013}) or to the characterization of duality for operator-induced function spaces (\cite{NekP2011,Miz2013}). The development of the theory itself, including a rather serious investigation of inequalities involving various modifications of supremum operators and their combinations with other types of operators also followed (see e.g.~\cite{Kre1,Kre2,GKPS,Kre4,GOLP} and many more).

Our point of departure will be a discrete version of a kernel operator. We say that $U\colon\Z\times\Z\to[0,\infty)$ is a \textit{regular kernel} if $U$ is nonincreasing in the first variable, nondecreasing in the second variable and there exists a positive constant $C$ such that
\begin{equation}\label{oin}
	U(i,n)\le C(U(i,j)+U(j,n)) \quad\text{for every $i,j,n\in\Z$, $i\leq j\leq n$}.
\end{equation}
We shall call $C$ from~\eqref{oin} the \textit{constant of regularity} of $U$.
Clearly, if $U$ is a regular kernel and $0 < p < \infty$, then $U^p$ is also a regular kernel.

Our definition of a regular kernel was introduced by Bloom and Kerman in~\cite{BK}. Other modifications of regular kernels also appear in the literature. For example, several authors define a~so-called Oinarov-type condition for a~kernel by replacing the monotonicity with a reverse inequality in~\eqref{oin}, which means that the kernel is quasimonotone. It is easy to see that these two definitions are equivalent in the sense that if a kernel satisfies the Oinarov condition, then there exists an equivalent kernel which is regular in the sense of Bloom and Kerman, and the constant of equivalence is absolute.

One of the central questions studied in this paper is the following. Given $0 < p,q \leq \infty$, a~regular kernel $U$ and two fixed sequences of nonnegative real numbers $\{v_n\}_{n\in\Z}$ and $\{w_n\}_{n\in\Z}$ (weights), we ask under which conditions there exists a positive constant $C$ such that
\begin{equation}\label{E:discrete-kernel-gop-dual-1}
	\Bigg(\sum_{n\in\Z}\Bigg(\sum_{i=-\infty}^n{U}(i,n)a_i\Bigg)^{q} {w}_n\Bigg)^{\frac 1q}
	\le C
	\Bigg(\sum_{n\in\Z}a_n\sp p{v}_n\Bigg)\sp{\frac 1p}
\end{equation}
holds for every sequence of nonnegative numbers $\{a_n\}_{n\in\Z}$ (as usual,~\eqref{E:discrete-kernel-gop-dual-1} has to be replaced by an~appropriate analog if either $p=\infty$ and/or $q=\infty$).
One can also study the dual inequality, namely
\begin{equation}\label{E:discrete-kernel-gop}
	\Bigg(\sum_{n\in\Z}\Bigg(\sum_{i=n}^{\infty}U(n,i)a_i\Bigg)^{q} w_n\Bigg)^{\frac 1q}
	\le C
	\Bigg(\sum_{n\in\Z}a_n\sp pv_n\Bigg)\sp{\frac 1p}.
\end{equation}
We note that~\eqref{E:discrete-kernel-gop} holds for every sequence of nonnegative numbers $\{a_n\}_{n\in\Z}$ if and only if
\begin{equation*}
	\Bigg(\sum_{n\in\Z}\Bigg(\sum_{i=-\infty}^{n}\overline{U}(i,n)a_i\Bigg)^{q} \overline{w}_n\Bigg)^{\frac 1q}
	\le C
	\Bigg(\sum_{n\in\Z}a_n\sp p\overline{v}_n\Bigg)\sp{\frac 1p}
\end{equation*}
holds for every sequence of nonnegative numbers $\{a_n\}_{n\in\Z}$, which is~\eqref{E:discrete-kernel-gop-dual-1} with appropriately modified kernel and weights. This follows on a simple index change $\overline{U}(i,n) = U(-n,-i)$, $\overline{v}_n = v_{-n}$ and $\overline{w}_n = w_{-n}$. It is not difficult to verify that the kernel $\overline{U}$ is also regular. So it is enough to investigate just~\eqref{E:discrete-kernel-gop-dual-1}. It turns out that the cases when $0 < p \leq 1$ and $1 \leq p < \infty$ are substantially different and have to be treated separately.

The paper is organized as follows. In the next section we shall collect results concerning Hardy-type operators containing kernels. In Section~\ref{S:3} we present results concerning supremum operators with kernels. Proofs are contained in the last two sections.

Throughout the paper, we shall denote by $\Z$ the set of all integers, by $\RpZ$ the collection of all double infinite sequences of nonnegative real numbers indexed over $\Z$ and by $\M_+$ the set of all nonnegative measurable functions on $\R$. Throughout the paper we denote $p'=\frac{p}{p-1}$ if $p\in (0,1)\cup (1,\infty)$. An~analogous notation is used for $q'$. We shall write $A\lesssim B$ if there exists a~positive constant $C$ such that $A \leq C B $. We write $A\approx B$ if we have simultaneously $A\lesssim B$ and $B\lesssim A$.

We shall throughout the paper use the notion of \textit{weight} both in a discrete setting, in which case it is just any sequence of nonnegative real numbers, and in the continuous one, in which case it is a nonnegative measurable function on an~appropriate subset of $\R$.

\section{Hardy type operators with kernels}\label{S:2}

There are several possibilities how to approach the study of~\eqref{E:discrete-kernel-gop-dual-1}. It can be for instance solved directly by using discretization and anti-discretization argument as it was done for the analogous continuous inequality by Bloom--Kerman~\cite{BK}, Oinarov~\cite{Oin}, Stepanov~\cite{Step} and K\v repela~\cite{Kre4}. We shall view the problem from a different perspective and point out that under the restriction $1 \leq p \leq \infty$ a characterization is available by simply bridging the discrete and the continuous worlds. The result is contained in the following theorem.

\begin{theorem}\label{T:main-prop1}
	Let $1 \leq p \leq \infty$ and $0 < q \leq \infty$. Let $U$ be a regular kernel on $\Z\times\Z$ and $\sqv,\sqw\in\RpZ$. We extend $U$ to a function defined on $\R^{2}$ by setting
	\begin{equation*}
		U(s,t)=\sum_{(i,n)\in\Z^2} U(i,n)\chi_{(i-1,i]\times(n-1,n]}(s,t) \quad\text{for $s,t\in \R$}
	\end{equation*}
	and	we define
	\begin{equation*}
		v(t)=\sum_{n\in\Z} v_n\chi_{(n-1,n]}(t),\quad
		w(t)=\sum_{n\in\Z} w_n\chi_{(n-1,n]}(t) \quad\text{for $t\in \R$}.
	\end{equation*}
	Then~\eqref{E:discrete-kernel-gop-dual-1} holds for some constant $C$ and every sequence $\sqa\in\RpZ$ if and only if there exists a constant $C'$ such that
	\begin{equation}\label{E:spoj-gop}
		\Bigg(\int_{-\infty}\sp{\infty}\Bigg(\int_{-\infty}^tU(s,t)f(s)\ds\Bigg)\sp qw(t)\dt\Bigg)\sp{\frac 1q}
		\le C'
		\Bigg(\int_{-\infty}\sp{\infty}f(t)\sp pv(t)\dt\Bigg)\sp{\frac 1p}
	\end{equation}
	holds for every $f\in\M_+$ (with natural analogues when either $p=\infty$ and/or $q=\infty$). Moreover, if $C, C'$ are optimal constants, then $C'\leq C\le 2^{1+\frac1q}C'$.
\end{theorem}

Using Theorem~\ref{T:main-prop1} and well-known necessary and sufficient conditions for~\eqref{E:spoj-gop}, we can now obtain a~characterization of~\eqref{E:discrete-kernel-gop-dual-1}.

\begin{theorem}\label{Th1.1}
	Let $1 \leq p \leq \infty$, $0 < q \leq \infty$. Let $U$ be a regular kernel and $\sqw,\sqv\in\RpZ $. Then the inequality \eqref{E:discrete-kernel-gop-dual-1} holds if and only if one of the following conditions holds:
	
	{\rm (i)} $p=1$, $0 < q < \infty$ and
	\begin{equation*}
		A_1:=\sup_{n\in \Z} \,v_n^{-1}\left(\sum_{i=n}^{\infty}U(n,i)^q w_i\right)^{\frac{1}{q}} < \infty;
	\end{equation*}
	
	{\rm (ii)} $p=1$, $q=\infty$ and
	\begin{equation*}
		A_2:=\sup_{n\in \Z} \,v_n^{-1}\left(\sup_{n\le i<\infty}U(n,i)w_i\right) < \infty;
	\end{equation*}
	
	{\rm (iii)} $p=\infty$, $ 1 \leq q < \infty$ and
	\begin{equation*}
		A_3:=\left(\sum_{n\in \Z} \,\left(\sum_{i=-\infty}^n U(i,n)v_i^{-1}\right)^q w_{n}\right)^{\frac{1}{q}}
		< \infty;
	\end{equation*}
	
	{\rm (iv)} $1 < p < \infty$, $q=1$ and
	\begin{equation*}
		A_4:=\left(\sum_{n\in \Z} \,\left(\sum_{i=n}^{\infty}U(n,i)w_i\right)^{p'} v_n^{1-p'}\right)^{\frac{1}{p'}}< \infty;
	\end{equation*}
	
	{\rm (v)} $1 < p < \infty$, $q=\infty$ and
	\begin{equation*}
		A_5:=\sup_{n\in \Z} w_n \,\left(\sum_{i=-\infty}^{n}
		U(i,n)^{p'} v_i^{1-p'}\right)^{\frac{1}{p'}}< \infty;
	\end{equation*}
	
	{\rm (vi)}  $p=q=\infty$ and
	\begin{equation*}
		A_6:=\sup_{n\in \Z} w_n \,\left(\sup_{-\infty<i\le n} U(i,n) v_i^{-1}\right)< \infty;
	\end{equation*}
	
	{\rm (vii)}  $1<p\le q<\infty$,
	\begin{equation*}
		A_7:=\sup_{n\in \Z} \,\left(\sum_{i=n}^{\infty}w_i\right)^{\frac{1}{q}}
		\left(\sum_{i=-\infty}^{n}
		U(i,n)^{p'} v_i^{1-p'}\right)^{\frac{1}{p'}}
		< \infty
	\end{equation*}
	and
	\begin{equation*}
		A_8:=\sup_{n\in \Z} \,\left(\sum_{i=n}^{\infty}U(n,i)^qw_i\right)^{\frac{1}{q}}
		\left(\sum_{i=-\infty}^{n}
		v_i^{1-p'}\right)^{\frac{1}{p'}}
		< \infty;
	\end{equation*}
	
	{\rm (viii)}  $1<q< p <\infty$,
	\begin{equation*}
		A_9:=\left(\sum_{n\in \Z} \,\left(\sum_{i=n}^{\infty}w_i\right)^{\frac{q}{p-q}} w_n
		\left(\sum_{i=-\infty}^{n}
		U(i,n)^{p'} v_i^{1-p'}\right)^{\frac{(p-1)q}{p-q}}
		\right)^{\frac{p-q}{pq}}< \infty
	\end{equation*}
	and
	\begin{equation*}
		A_{10}:=\left(\sum_{n\in \Z} \,\left(\sum_{i=n}^{\infty}U(n,i)^qw_i\right)^{\frac{p}{p-q}} v_n^{1-p'} \left(\sum_{i=-\infty}^{n}  v_i^{1-p'}\right)^{\frac{p(q-1)}{p-q}}
		\right)^{\frac{p-q}{pq}}
		< \infty;
	\end{equation*}
	
	{\rm (ix)}  $1< p <\infty$, $0<q<p$, $A_9< \infty$ and
	\begin{equation*}
		A_{11}:=\left(\sum_{n\in \Z} \,\left(\sum_{i=n}^{\infty}U(n,i)^q w_i\right)^{\frac{q}{p-q}} w_n \sup_{-\infty < j \le n}U(j,n)^{q}
		\left(\sum_{i=-\infty}\sp{j}
		v_i^{1-p'}\right)^{\frac{(p-1)q}{p-q}}
		\right)^{\frac{p-q}{pq}}
		< \infty;
	\end{equation*}
	
	{\rm (x)}  $0<q<p=1$,
	\begin{equation*}
		A_{12}:=\left(\sum_{n\in \Z} \,\left(\sum_{i=n}^{\infty}w_i\right)^{-q'} w_n
		\sup_{-\infty < i\le n}
		U(i,n)^{-q'} v_i^{q'}\right)^{-\frac{1}{q'}}
		< \infty
	\end{equation*}
	and
	\begin{equation*}
		A_{13}:=\left(\sum_{n\in \Z} \,\left(\sum_{i=n}^{\infty}U(n,i)^q w_i\right)^{-q'} w_n \sup_{-\infty<i\le n}U(i,n)^{q}
		v_i^{q'}\right)^{-\frac{1}{q'}}< \infty.
	\end{equation*}
	
	Moreover, if $C$ is the optimal constant in~\eqref{E:discrete-kernel-gop-dual-1}, then, in cases \textup{(i)}--\textup{(vi)}, $C\approx A_k$, where $k$ is the appropriate index corresponding to the label of the particular case, and in cases \textup{(vii)}-- \textup{(x)}, we have  $C\approx A_k+A_{\ell}$, where $k$ and $\ell$ are appropriate indices  \textup{(}$C\approx A_7+A_8$ in  \textup{(}vii\textup{)} and so on\textup{)}.
\end{theorem}

In the case when $0 < p < 1$ the above approach does not work. It is a~notoriously known fact that in this case the discrete inequality may perfectly hold while its continuous analogue is impossible unless the weights are trivial. We shall eventually show that a~certain type of bridge is nevertheless still available. We begin by observing that~\eqref{E:discrete-kernel-gop-dual-1} is equivalent to two other inequalities, different in nature.

\begin{theorem}\label{T:kernel-main}
	Let $0 < p \leq 1$ and $0 < q \leq \infty$. Let $U$ be a regular kernel on $\Z\times\Z$ and $\sqv,\sqw\in\RpZ$. Then the following three statements are equivalent.
	
	\textup{(i)} There exists a constant $C$ such that~\eqref{E:discrete-kernel-gop-dual-1} holds for every sequence $\{a_n\} \in\RpZ$.
	
	\textup{(ii)}There exists a constant $C'$ such that
	\begin{equation}\label{E:discrete-kernel-weak}
		\Bigg(\sum_{n\in\Z}\Bigg(\sup_{-\infty<i\le n}U(i,n)a_i\Bigg)^{q} w_n\Bigg)^{\frac 1q}
		\le C'
		\Bigg(\sum_{n\in\Z}a_n\sp pv_n\Bigg)\sp{\frac 1p}
	\end{equation}
	holds for every sequence $\{a_n\}\in\RpZ$.
	
	\textup{(iii)} There exists a constant $C''$ such that
	\begin{equation}\label{E:discrete-kernel-strong}
		\Bigg(\sum_{n\in\Z}\Bigg(\sum_{i=-\infty}^{n} U(i,n)^{p}a_i\Bigg)^{\frac{q}{p}} w_n\Bigg)^{\frac pq}
		\le C''
		\sum_{n\in\Z}a_nv_n.
	\end{equation}
	holds for every sequence $\{a_n\} \in\RpZ$.
	
	Moreover, if $C, C'$ and $C''$ are the best constants in~\eqref{E:discrete-kernel-gop-dual-1}, \eqref{E:discrete-kernel-weak} and \eqref{E:discrete-kernel-strong}, respectively, then $ C\approx C'\approx (C'')^{\frac1p}$.
\end{theorem}
It is worth noticing that the assertion of Theorem~\ref{T:kernel-main} cannot be extended to the case when $1 < p \leq \infty$.

We can also formulate the dual statement to Theorem~\ref{T:kernel-main}.

\begin{theorem}
	Let $0 < p \leq 1 $ and $0 < q \leq \infty$. Let $U$ be a regular kernel on $\Z\times\Z$ and $\sqv,\sqw\in\RpZ$. Then the following three statements are equivalent.
	
	\textup{(i)} There exists a constant $C$ such that~\eqref{E:discrete-kernel-gop} holds for every sequence $\{a_n\} \in\RpZ$.
	
	\textup{(ii)}There exists a constant $C'$ such that
	\begin{equation}\label{E:discrete-kernel-gop-1}
		\Bigg(\sum_{n\in\Z}\Bigg(\sup_{n\leq i<\infty}U(n,i)a_i\Bigg)^{q} w_n\Bigg)^{\frac 1q}
		\le C'
		\Bigg(\sum_{n\in\Z}a_n\sp pv_n\Bigg)\sp{\frac 1p}
	\end{equation}
	holds for every sequence $\{a_n\} \in\RpZ$.
	
	\textup{(iii)} There exists a constant $C''$ such that
	\begin{equation} \label{E:discrete-kernel-gop-2}
		\Bigg(\sum_{n\in\Z}\Bigg(\sum_{i=n}^{\infty} U(n,i)^{p}a_i\Bigg)^{\frac{q}{p}} w_n\Bigg)^{\frac pq}
		\le C''
		\sum_{n\in\Z}a_nv_n.
	\end{equation}
	holds for every sequence $\{a_n\} \in\RpZ$.
	
	Moreover, if $C,\, C'$ and $C''$ are the best constants in~\eqref{E:discrete-kernel-gop}, \eqref{E:discrete-kernel-gop-1} and \eqref{E:discrete-kernel-gop-2}, respectively, then $ C\approx C'\approx (C'')^{\frac1p}$.
\end{theorem}
We are now in a position to characterize~\eqref{E:discrete-kernel-gop-dual-1} by a~continuous-type condition. To do this we shall use the fact that $p=1$ is a meeting point of both the cases and hence either approach works in this case.

\begin{theorem}\label{T:main-prop-dual-1}
	Let $0 < p \leq 1$ and $0 < q < \infty$. Let $U$, $\{v_n\}$, $\{w_n\}$, $v$ and $w$ be as in Theorem~\ref{T:main-prop1}. Then~\eqref{E:discrete-kernel-gop-dual-1} holds for some $C$ and every sequence $\sqa\in\RpZ$ if and only if there is a $C'$ such that
	\begin{equation}\label{E:kernel-gop-dual-1}
		\Bigg(\int_{-\infty}\sp{\infty}\Bigg(\int_{-\infty}^tU(s,t)^pf(s)\ds\Bigg)\sp {\frac{q}{p}}w(t)\dt\Bigg)\sp{\frac pq}
		\le C'
		\int_{-\infty}\sp{\infty}f(t)v(t)\dt
	\end{equation}
	holds for every $f\in\M_+$.
	
	Moreover, if $C$ and $C'$ are the best constants in~\eqref{E:discrete-kernel-gop-dual-1} and  \eqref{E:kernel-gop-dual-1}, respectively, then $ C\approx (C')^{\frac1p}$.
	
	Similarly, \eqref{E:discrete-kernel-gop} holds for some $C$ and every sequence $\sqa\in\RpZ$ if and only if there is a $C''$ such that
	\begin{equation}\label{E:kernel-gop-dual-2}
		\Bigg(\int_{-\infty}\sp{\infty}\Bigg(\int_t^{\infty}U(t,s)^pf(s)\ds\Bigg)\sp {\frac qp}w(t)\dt\Bigg)\sp{\frac pq}
		\le C''
		\int_{-\infty}\sp{\infty}f(t)v(t)\dt
	\end{equation}
	holds for every $f\in\M_+$.
	
	Moreover, if $C$ and $C''$ are the best constants in~\eqref{E:discrete-kernel-gop}  and \eqref{E:kernel-gop-dual-2}, respectively, then $ C\approx (C'')^{\frac1p}$.
\end{theorem}

Now we can characterize the discrete inequality~\eqref{E:discrete-kernel-gop-dual-1} for $0 < p \leq 1$ by conditions appearing in Theorem~\ref{Th1.1}.

\begin{theorem}\label{T:only-p-1}
	Let $0 < p \leq 1$ and $0 < q \leq \infty$. Let $U$ be a regular kernel and $\sqw,\sqv\in \RpZ $. The inequality \eqref{E:discrete-kernel-gop-dual-1} holds if and only if one of the following conditions holds:
	
	{\rm (i)} $0<p\le q<\infty$ and $A_1<\infty$;
	
	{\rm (ii)}  $q=\infty$ and $A_2<\infty$;
	
	{\rm (iii)}  $0<q<p$, $A_{12}<\infty$ and $A_{13}<\infty$.
	
	Moreover, if $C$ is the best constant in~\eqref{E:discrete-kernel-gop-dual-1}, then $C$ is equivalent to $A_1$ in the case {\rm (i)}, to $A_2$  in the case {\rm (ii)} and to $A_{12}+A_{13}$  in the case {\rm (iii)}.
\end{theorem}

\section{Supremum operators with kernels}\label{S:3}
We shall now turn our attention to inequalities which involve operators in which both a kernel and the supremum are combined. Namely, we are interested under what conditions under parameters $0 < p,q \leq \infty$, a regular kernel $U$ and two fixed sequences $\{w_n\}, \{v_n\}\in\RpZ$ there exists a~positive constant $C$ such that the inequality
\begin{equation}\label{E:main-supremal-small}
	\Bigg( \sumZ w_n \Big[ \sup_{-\infty<i\le n} U(i,n) \sum_{j=-\infty}^{i} a_j \Big]^q \Bigg)^\jq
	\le C
	\Bigg( \sumZ a_n^p v_n \Bigg)^\jp
\end{equation}
holds for every sequence $\{a_n\}\in\RpZ$.

\begin{theorem}\label{T:six}
	Let	$0 < p \leq 1$, $ 0 < q < \infty$ and $\sqv,\sqw\in\RpZ$. Let $U$ be a regular kernel. Define
	\begin{align*}
		B_1 &= \sup_{\sqa\in\RpZ} \Bigg( \sumZ w_n \Big[ \sup_{-\infty<i\le n}\,U(i,n)    a_i \Big]^q \Bigg)^\jq \Bigg( \sumZ a_n^p v_n \Bigg)^\mjp, \\
		B_2 &= \sup_{\sqa\in\RpZ} \Bigg( \sumZ w_n \Big[ \sup_{-\infty<i\le n}\,U(i,n)   \sup_{-\infty<j\le i} a_j \Big]^q \Bigg)^\jq \Bigg( \sumZ a_n^p v_n \Bigg)^\mjp, \\
		B_3 &= \sup_{\sqa\in\RpZ} \Bigg( \sumZ w_n \Big[ \sup_{-\infty<i\le n} U(i,n) \sum_{j=-\infty}^{i} a_j \Big]^q \Bigg)^\jq \Bigg( \sumZ a_n^p v_n \Bigg)^\mjp, \\
		B_4 &= \sup_{\sqa\in\RpZ} \Bigg( \sumZ w_n \Big[ \sum_{i=-\infty}^{n} U(i,n) a_i \Big]^q \Bigg)^\jq \Bigg( \sumZ a_n^p v_n \Bigg)^\mjp, \\
		B_5 &= \sup_{\sqa\in\RpZ} \Bigg( \sumZ w_n \Big[ \sup_{-\infty<i\le n} U(i,n)^p \sum_{j=-\infty}^{i} a_j^p \Big]^\qp \Bigg)^\jq \Bigg( \sumZ a_n^p v_n \Bigg)^\mjp, \\
		B_6 &= \sup_{\sqa\in\RpZ} \Bigg( \sumZ w_n \Big[ \sum_{i=-\infty}^{n} U(i,n)^p a_i^p \Big]^\qp \Bigg)^\jq \Bigg( \sumZ a_n^p v_n \Bigg)^\mjp.
	\end{align*}
	Then $B_1$, $B_2$, $B_3$, $B_4$, $B_5$ and $B_6$ are mutually equivalent, and, moreover, the equivalence constants depend only on $p$ and $q$.	
\end{theorem}

We will also present a dual version of these results.

\begin{theorem}
	Let	$0 < p \leq 1$, $0 < q < \infty$ and $\sqv,\sqw\in\RpZ$. Let $U$ be a regular kernel. Define
	\begin{align*}
		\tilde B_1 &= \sup_{\sqa\in\RpZ} \Bigg( \sumZ w_n \Big[ \sup_{n \le i <\infty }\,U(n,i)    a_i \Big]^q \Bigg)^\jq \Bigg( \sumZ a_n^p v_n \Bigg)^\mjp, \\
		\tilde B_2 &= \sup_{\sqa\in\RpZ} \Bigg( \sumZ w_n \Big[ \sup_{n\le i < \infty}\,U(n,i)   \sup_{i\le j<\infty} a_j \Big]^q \Bigg)^\jq \Bigg( \sumZ a_n^p v_n \Bigg)^\mjp, \\
		\tilde B_3 &= \sup_{\sqa\in\RpZ} \Bigg( \sumZ w_n \Big[ \sup_{n\le i<\infty } U(n,i) \sum_{j=i}^{\infty} a_j \Big]^q \Bigg)^\jq \Bigg( \sumZ a_n^p v_n \Bigg)^\mjp, \\
		\tilde B_4 &= \sup_{\sqa\in\RpZ} \Bigg( \sumZ w_n \Big[ \sum_{i=n}^{\infty} U(n,i) a_i\Big]^q \Bigg)^\jq \Bigg( \sumZ a_n^p v_n \Bigg)^\mjp, \\
		\tilde B_5 &= \sup_{\sqa\in\RpZ} \Bigg( \sumZ w_n \Big[ \sup_{n\le i<\infty} U(n,i)^p \sum_{j=i}^{\infty} a_j^p \Big]^\qp \Bigg)^\jq \Bigg( \sumZ a_n^p v_n \Bigg)^\mjp, \\
		\tilde B_6 &= \sup_{\sqa\in\RpZ} \Bigg( \sumZ w_n \Big[ \sum_{i=n}^{\infty} U(n,i)^p a_i^p \Big]^\qp \Bigg)^\jq \Bigg( \sumZ a_n^p v_n \Bigg)^\mjp.
	\end{align*}
	Then $\tilde B_1$, $\tilde B_2$, $\tilde B_3$, $\tilde B_4$, $\tilde B_5$ and $\tilde B_6$ are mutually equivalent, and, moreover, the equivalence constants depend only on $p$ and $q$.	
\end{theorem}

Given a~sequence of nonnegative real numbers $\squ$, consider the kernel $U$, given by $U(i,n)=u_i$ for every $i,n\in\Z$. This kernel is of course not regular in general, since we do not assume any monotonicity of $\squ$, but we can obtain a~result in the spirit of Theorem~\ref{T:six} for it nevertheless.

\begin{theorem}\label{T:hux}
	Let	$0 < p \leq 1$, $0 < q < \infty$ and $\{u_n\}, \sqv,\sqw\in\RpZ$. Define
	\begin{align*}
		\mathcal B_1 &= \sup_{\sqa\in\RpZ} \Bigg( \sumZ w_n \Big[ \sup_{-\infty<i\le n}\,u_i   \sup_{-\infty<j\le i} a_j \Big]^q \Bigg)^\jq \Bigg( \sumZ a_n^p v_n \Bigg)^\mjp, \\
		\mathcal B_{2} &= \sup_{\sqa\in\RpZ} \Bigg( \sumZ w_n \Big[ \sup_{-\infty<i\le n} (\sup_{i\le j\le n} u_j ) a_i \Big]^q \Bigg)^\jq \Bigg( \sumZ a_n^p v_n \Bigg)^\mjp, \\
		\mathcal B_3 &= \sup_{\sqa\in\RpZ} \Bigg( \sumZ w_n \Big[ \sup_{-\infty<i\le n} u_i \sum_{j=-\infty}^{i} a_j \Big]^q \Bigg)^\jq \Bigg( \sumZ a_n^p v_n \Bigg)^\mjp, \\
		\mathcal B_4 &= \sup_{\sqa\in\RpZ} \Bigg( \sumZ w_n \Big[ \sup_{-\infty<i\le n}(\sup_{i\le j\le n} u_j ) \sum_{j=-\infty}^{i} a_j \Big]^q \Bigg)^\jq \Bigg( \sumZ a_n^p v_n \Bigg)^\mjp, \\
		\mathcal B_5 &= \sup_{\sqa\in\RpZ} \Bigg( \sumZ w_n \Big[ \sum_{i=-\infty}^{n} (\sup_{i\le j\le n} u_j ) a_i \Big]^q \Bigg)^\jq \Bigg( \sumZ a_n^p v_n \Bigg)^\mjp, \\
		\mathcal B_6 &= \sup_{\sqa\in\RpZ} \Bigg( \sumZ w_n \Big[ \sup_{-\infty < i\le n}  u_i^p \sum_{j=-\infty}^{i} a_j^p \Big]^\qp \Bigg)^\jq \Bigg( \sumZ a_n^p v_n \Bigg)^\mjp, \\
		\mathcal B_7 &= \sup_{\sqa\in\RpZ} \Bigg( \sumZ w_n \Big[ \sup_{-\infty<  i\le n}  (\sup_{i\le j\le n } u_j^p ) \sum_{j=-\infty}^{i} a_j^p \Big]^\qp \Bigg)^\jq \Bigg( \sumZ a_n^p v_n \Bigg)^\mjp, \\
		\mathcal B_8&= \sup_{\sqa\in\RpZ} \Bigg( \sumZ w_n \Big[ \sum_{i=-\infty}^{n} (\sup_{i\le j\le n } u_j^p ) a_i^p \Big]^\qp \Bigg)^\jq \Bigg( \sumZ a_n^p v_n \Bigg)^\mjp.
	\end{align*}
	Then $\mathcal B_1$, $\mathcal B_2$, $\mathcal B_3$, $\mathcal B_4$, $\mathcal B_5$, $\mathcal B_6$, $\mathcal B_7$ and $\mathcal B_8$ are mutually equivalent, and, moreover, the equivalence constants depend only on $p$ and $q$.	
\end{theorem}

\begin{theorem}\label{T:cd}
	Let $0 < p \leq 1$ and $0 < q \leq \infty $. Let $U$ be a regular kernel and $\sqw,\sqv\in \RpZ $.
	Then the inequality~\eqref{E:main-supremal-small} holds for every sequence $\sqa\in \RpZ$ if and only if one of the conditions {\rm (i)---(iii)} of Theorem~\ref{T:only-p-1} is satisfied.
\end{theorem}

In the case $1 \leq p \leq \infty$ the situation is much more delicate than in the case $0 < p \leq 1$. We can still get a~bridge type equivalence theorem in the spirit of Theorem~\ref{T:main-prop1}, but this time there is no characterization available of the integral inequality. The proof is the same as that of Theorem~\ref{T:main-prop1}, and hence we omit it.

\begin{theorem}\label{T:main-supremal}
	Let $1 \leq p \leq \infty$ and $0 < q \leq \infty$. Let $U$, $\{v_n\}$, $\{w_n\}$, $v$ and $w$ be as in Theorem~\ref{T:main-prop1}.
	Then~\eqref{E:main-supremal-small} holds for every sequence $\sqa\in\RpZ$ if and only if there is a~constant $C_1$ such that
	\begin{equation} \label{1-bis}
		\left(\int_{-\infty}^\infty\left(\esssup_{-\infty < y\leq x}U(y,x)\int_{-\infty}^y f(t)\,dt\right)^q w(x)dx\right)^{\frac{1}{q}}\leq C_1\left(\int_{-\infty}^\infty f(t)^pv(t)\,dt\right)^{\frac{1}{p}}
	\end{equation}
	holds for every $f\in\M_+$ (with natural analogues when either $p=\infty$ or $q=\infty$).
	
	Moreover, if $C$ and $C_1$ are the best constants in~\eqref{E:main-supremal-small} and \eqref{1-bis}, respectively, then $ C\approx C_1$.
\end{theorem}

Now we shall point out a simple but extremely useful fact that weighted inequalities for the Hardy operator have certain scaling property. Let us note that similar phenomenon has been observed before in connection with the geometric mean operator (see~\cite{PO-geom}). We shall first need some notation.

Let $v$ and $w$ be weight functions on $\R$. We denote
\begin{equation*}
	\sigma_p(a,b)= \left\{
	\begin{array}{cc}
		\bigg(\int_a^b v^{1-p'}(y) dy \bigg)^{\frac{1}{p'}}	&\text{for $1 < p < \infty$},  \\
		\esssup\limits_{a<y<b} v(y)^{-1}	& \text{for $p=1$}.
	\end{array}
	\right.
\end{equation*}

The following theorem was proved in \cite[Theorem 3.2]{Si-St96} when $p=1$.

\begin{theorem}\label{T:cont}
	Let $1 \le p < \infty$, $0 < q < \infty$ and let $w$ be a~weight on $\R$. Then the following two statements are equivalent.
	
	\textup{(i)} There exists a constant $C$, depending on $p,q,v$ and $w$ such that
	\begin{equation}\label{1-}
		\left(\int_{-\infty}^\infty w(x) \left(\int_{-\infty}^x f(t)\,dt \right)^q dx\right)^{\frac{1}{q}} \leq C
		\left(\int_{-\infty}^\infty f(t)^p v(t)\,dt\right)^{\frac{1}{p}}
	\end{equation}
	holds for every $f\in\M_+$.
	
	\textup{(ii)} There exists a constant $C'$ depending on $p,q,v$ and $w$ such that
	\begin{equation}\label{2-}
		\left(\int_{-\infty}^\infty w(x) \left(\int_{-\infty}^x f(t)\,dt \right)^{\frac{q}{p}} dx\right)^{\frac{1}{q}} \leq C'
		\left(\int_{-\infty}^\infty f(x) \sigma_p({-\infty},x)^{-p}\,dx\right)^{\frac{1}{p}}
	\end{equation}
	holds for every $f\in\M_+$.
	
	Moreover, if $C$ and $C'$ are the best constants in~\eqref{1-} and \eqref{2-}, respectively, then $ C\approx C'$.
\end{theorem}

\begin{theorem}\label{disc. equiv.}
	Let $1 \le p < \infty$, $0 < q < \infty$, $\sqb,\sqc\in\RpZ$ and $C>0$. Then the  following two statements are equivalent.
	
	\textup{(i)} There exists a positive constant $C$ depending only on $p,q,\sqb$ and $\sqc$ such that
	\begin{equation}\label{3-}
		\left(\sum_{k\in {\mathbb Z}} b_k \left(\sum_{i=-\infty}^{k} a_i c_i \right)^q \right)^{\frac{1}{q}}
		\leq C \left(\sum_{k\in {\mathbb Z}} a_i^p\right)^{\frac{1}{p}}
	\end{equation}
	holds for every $\sqa\in\RpZ$.
	
	\textup{(ii)} There exists a positive constant $C'$ depending only on $p,q,\sqb$ and $\sqc$ such that either the inequality
	\begin{equation}\label{4-}
		\left(\sum_{k\in {\mathbb Z}} b_k \left(\sum_{i=-\infty}^{k} a_i \left(\sum_{j=-\infty}^{i} c_j^{p'}
		\right)^{\frac{p}{p'}}\right)^{\frac{q}{p}} \right)^{\frac{1}{q}}
		\leq C' \left(\sum_{k\in {\mathbb Z}}a_i\right)^{\frac{1}{p}}	
	\end{equation}
	holds for every $\sqa\in\RpZ$ (when $1<p<\infty$), or such that inequality
	\begin{equation*}
		\left(\sum_{k\in {\mathbb Z}} b_k \left(\sum_{i=-\infty}^{k} a_i \sup_{-\infty< j\le i} c_j
		\right)^{q} \right)^{\frac{1}{q}}
		\leq C' \sum_{k\in {\mathbb Z}}
		a_i
	\end{equation*}
	holds for every $\sqa\in\RpZ$ (when $p=1$).
	
	Moreover, if $C$ and $C'$ are the best constants in \eqref{3-} and \eqref{4-}, respectively, then $C \approx C'$.
\end{theorem}

The proof is the same as that of Theorem~\ref{T:cont} and therefore is omitted.

Now we shall formulate the main result of this section. As was already mentioned, no characterization is known in the case $1 \leq p < \infty$ and $0<q<\infty$ of the inequality~\eqref{1-bis}. It is worth noticing that for $p=\infty$ and $0 < q \leq \infty$ the characterization is trivial, while for $q = \infty$ and any $p$, \eqref{1-bis} turns easily to a classical Hardy inequality after interchanging of the essential suprema. Our next theorem shows that, nevertheless, an application of the scaling argument to supremum operators leads to a rather surprising characterization of~\eqref{1-bis} for $1 \leq p<\infty$ and $0<q< \infty$.

\begin{theorem} \label{main}
	Let $1 \leq p<\infty$, $0<q< \infty$, let $v,w$ be weights on $\R$ such that $\int_{x}^{\infty} w(t) dt <  \infty$ for every $x\in\R$ and let $U$ be a regular kernel. Then the following three statements are equivalent.
	
	\textup{(i)} There exists a~constant $C_1$ such that~\eqref{1-bis} holds for every $f\in\mathcal{M}_+$.
	
	\textup{(ii)} There exists a constant $C_2$ such that
	\begin{equation}\label{2-bis}
		\left(\int_{-\infty}^\infty\left(\int_{-\infty}^x
		U(y,x)^p f(y)dy\right)^{\frac{q}{p}}w(x)dx\right)^{\frac{p}{q}}\leq C_2 \int_{-\infty}^\infty
		f(x)\sigma_p({-\infty},x)^{-p} dx
	\end{equation}
	holds for every $f\in\mathcal{M}_+$.
	
	\textup{(iii)} There exists a constant $C_3$ such that
	\begin{equation} \label{3-bis}
		\left(\int_{-\infty}^\infty\left(\esssup_{-\infty< y \leq x} U(y,x)^p\int_{-\infty}^y f(t)\,dt \right)^{\frac{q}{p}} w(x) dx \right)^{\frac{p}{q}} \leq C_3 \int_{-\infty}^\infty f(x)	\sigma_p({-\infty},x)^{-p} dx
	\end{equation}
	holds for every $f\in\mathcal{M}_+$.
	
	Moreover, if $C_1,C_2$ and $C_3$ are the best constants in~\eqref{1-bis}, \eqref{2-bis} and \eqref{3-bis}, respectively, then $ C_1\approx C_2^{\frac{1}{p}}\approx C_3^{\frac{1}{p}}$.
\end{theorem}

The proof of this theorem will be given in the last section.

Using known characterizations  of the inequality \eqref{2-bis} that have been already mentioned (see the remarks preceding Theorem~\ref{T:main-prop1}) and Theorem~\ref{main}, we can now obtain a characterization of the inequality \eqref{1-bis}.

\begin{theorem}
	Let $1 \leq p\le \infty$ and $0<q\le  \infty$  and $U$ be a regular kernel. Then inequality~\eqref{1-bis} holds if and only if
	
	{\rm (i)} $1\le p \le q < \infty$,
	\begin{equation*}
		\mathcal A_1:=\sup_{x\in \R} \,\sigma_p({-\infty},x)\left(\int_{x}^{\infty}U(x,y)^q w(y)\dy\right)^{\frac{1}{q}} < \infty;
	\end{equation*}
	
	{\rm (ii)} $1\le p < q= \infty$,
	\begin{equation*}
		\mathcal A_2:=\sup_{x\in \R} \,\sigma_p({-\infty},x)\left(\esssup_{x\le y<\infty}U(x,y)w(y)\right) < \infty;
	\end{equation*}
	
	{\rm (iii)} $p=q=\infty$,
	\begin{equation*}
		\mathcal A_3:=\sup_{x\in \R} \,v^{-1}(x)\left(\esssup_{x\le y<\infty}U(x,y)w(y)\right) < \infty;
	\end{equation*}
	
	{\rm (iv)} $0<q<p=\infty$,
	\begin{equation*}
		\mathcal A_4:=\left(\int_{{-\infty}}^{\infty} \,\left(\int_{-\infty}^{x} U(y,x) v^{-1}(y)\dy\right)^q w(x)\dx\right)^{\frac{1}{q}}
		< \infty;
	\end{equation*}

	{\rm (v)}  $0<q<p$, $1\le p <\infty$,
	\begin{equation*}
		\mathcal A_{12}:=\left(\int_{-\infty}^{\infty} \,\left(\int_{x}^{\infty}w(y)\dy\right)^{\frac{q}{p-q}} w(x)
		\esssup_{-\infty <y\le x}
		U(y,x)^{\frac{pq}{p-q}} \sigma_p({-\infty},x)^{\frac{q}{p-q}} \dx \right)^{\frac{p-q}{pq}}
		< \infty,
	\end{equation*}
	\begin{equation*}
		\mathcal A_{13}:=\left(\int_{-\infty}^{\infty} \,\left(\int_{x}^{\infty}U(x,y)^q w(y)\dy \right)^{\frac{q}{p-q}} w(x) \esssup_{{-\infty}<y\le x}U(y,x)^{q}
		\sigma_p(-\infty,x)^{\frac{q}{p-q}} \, \dx \right)^{\frac{p-q}{pq}}< \infty.
	\end{equation*}
	Moreover, if $C$ is the optimal constant in~\eqref{1-bis}, then, in cases \textup{(i)}--\textup{(iv)}, $C\approx \mathcal A_k$, where $k$ is the index corresponding to the label of the case, and in case \textup{(v)} we have  $C\approx \mathcal A_{12}+\mathcal A_{13}$.
\end{theorem}

Our next aim is to use the results obtained to characterize the discrete inequality, namely \eqref{E:main-supremal-small}, by means of discrete conditions. We shall use the following notation.

For every $\sqv\in\RpZ$, $N,M \in \Z$, $N<M$,  define
\begin{equation*}
	\sigma_p(N,M):= \left\{
	\begin{array}{cc}
		\bigg(\sum_{i=N}^M v_i^{1-p'} \bigg)^{\frac{1}{p'}},	& 1 < p< \infty,  \\
		\sup\limits_{N\le i\le M} v_i^{-1},	& p=1.
	\end{array}
	\right.
\end{equation*}

\begin{theorem}\label{T:supremalpge}
	Let	$1 \leq p < \infty$, $ 0 < q < \infty$ and $\sqv,\sqw\in\RpZ$. Let $U$ be a regular kernel. Then the following three conditions are equivalent.
	
	\textup{(i)} There exists a~constant $C$ such that~\eqref{E:main-supremal-small} holds for every sequence $\sqa\in\RpZ$.
	
	\textup{(ii)} There exists a constant $C'$ such that
	\begin{equation}\label{E:C'}
		\Bigg( \sumZ w_n \Big[ \sum_{i=-\infty}^{n} U(i,n)^p a_i^p \Big]^\qp \Bigg)^\jq \le C' \Bigg( \sumZ \sigma_p(-\infty,n)^{-p} a_n^p \Bigg)^\mjp
	\end{equation}
	holds for every sequence $\sqa\in\RpZ$.
	
	\textup{(iii)} There exists a constant $C''$ such that
	\begin{equation}\label{E:C''}
		\Bigg( \sumZ w_n \Big[ \sup_{-\infty<i\le n} U(i,n)^p \sum_{j=-\infty}^{i} a_j^p \Big]^\qp  \Bigg)^\jq\le C'' \Bigg( \sumZ \sigma_p(-\infty,n)^{-p} a_n^p \Bigg)^\mjp
	\end{equation}
	holds for every sequence $\sqa\in\RpZ$.
	
	Moreover, if $C,C'$ and $C''$ are the best constants in~\eqref{E:main-supremal-small}, \eqref{E:C'} and~\eqref{E:C''}, respectively, then $C\approx C'\approx C''$.
\end{theorem}

Now we shall characterize~\eqref{E:main-supremal-small} by discrete conditions.

\begin{theorem}\label{T:supremalpge1}
	Let $1 \leq p \leq \infty$ and $0 < q \leq \infty$. Let $U$ is a regular kernel and $\sqw,\sqv\in \RpZ $. The inequality \eqref{E:main-supremal-small} holds if and only if
	
	{\rm (i)} $1\le p\le q<\infty$,
	\begin{equation*}
		D_1:=\sup_{n\in \Z} \,\sigma_p(-\infty,n)\left(\sum_{i=n}^{\infty}U(n,i)^q w_i\right)^{\frac{1}{q}} < \infty;
	\end{equation*}
	
	{\rm (ii)}  $1\le p <q=\infty$,
	\begin{equation*}
		D_2:=\sup_{n\in \Z} \,\sigma_p(-\infty,n)\left(\sup_{n\le i<\infty}U(n,i)w_i^{\frac{1}{p}}\right) < \infty;
	\end{equation*}
	
	{\rm (iii)}  $ p=q=\infty$,
	\begin{equation*}
		D_3:=\sup_{n\in \Z} \,v_n^{-1}\left(\sup_{n\le i<\infty}U(n,i)w_i^{\frac{1}{p}}\right) < \infty;
	\end{equation*}
	
	{\rm (iv)}  $0<q<p=\infty$,
	\begin{equation*}
		D_4:=\left(\sum_{n\in \Z} \,\left(\sum_{i=-\infty}^n U(i,n)v_i^{-1}\right)^q w_{n}\right)^{\frac{1}{q}}
		< \infty;
	\end{equation*}
	
	{\rm (v)}  $0<q<p<\infty$,
	\begin{equation*}
		D_{5}:=\left(\sum_{n\in \Z} \,\left(\sum_{i=n}^{\infty}w_i\right)^{\frac{q}{p-q}} w_n
		\sup_{-\infty<i\le n}
		U(i,n)^{-\frac{qp}{q-p}} \sigma_p(-\infty,i)^{\frac{q}{q-p}}\right)^{\frac{p-q}{pq}}
		< \infty,
	\end{equation*}
	\begin{equation*}
		D_{6}:=\left(\sum_{n\in \Z} \,\left(\sum_{i=n}^{\infty}U(n,i)^q w_i\right)^{\frac{q}{p-q}} w_n \sup_{-\infty<i\le n}U(i,n)^q
		\sigma_p(-\infty,i)^{\frac{q}{q-p}}\right)^{\frac{p-q}{pq}}< \infty.
	\end{equation*}
	Moreover, if $C$ is the optimal constant in~\eqref{E:main-supremal-small}, then, in cases \textup{(i)}--\textup{(iv)}, $C\approx  D_k$, where $k$ is the index corresponding to the label of the case, and in case \textup{(v)} we have  $C\approx  D_{5}+ D_{6}$.
\end{theorem}

\section{Proofs of theorems from Section~\ref{S:2}}

\begin{proof}[Proof of Theorem~\ref{T:main-prop1}]
	We shall perform the proof only in the case when both $p$ and $q$ are finite. The proof in the remaining cases is analogous and therefore omitted. Suppose that~\eqref{E:discrete-kernel-gop-dual-1} holds and let $f\in\MM$. Set $a_n=\int_{n-1}^{n} f$   for $n\in\Z$. Then we get, using the H\"older inequality,
	\begin{equation*}
		\left(\sum_{n\in\Z}a_n^p v_n\right)^\jp
		\le
		\left(\sum_{n\in\Z}\int_{n-1}^{n} f(t)^p v_n\dt \right)\sp{\frac 1p}
		=
		\left(\int_{-\infty}\sp{\infty}f(t)^p v(t)\dt\right)^\jp
	\end{equation*}
	and
	\begin{align*}
		\left(
		\sum_{n\in\Z}\left(\sum_{i=-\infty}^{n} U(i,n) a_i\right)^q w_n
		\right)\sp{\frac 1q}
		&=\left(
		\sum_{n\in\Z} \left(\int_{-\infty}^{n}U(y,n)f(y)\dy\right)^q \int_{n-1}^{n} w(t)\dt
		\right)\sp{\frac 1q}  \\
		&\ge\left(
		\sum_{n\in\Z}\int_{n-1}^{n} \left(\int_{-\infty}^{t} U(y,t)f(y)\dy\right)^q w(t)\dt
		\right)\sp{\frac 1q} \\
		&=\left(
		\int_{-\infty}^{\infty} \left(\int_{-\infty}^{t} U(y,t)f(y)\dy\right)^q w(t)\dt
		\right)\sp{\frac 1q},
	\end{align*}
	and~\eqref{E:spoj-gop} follows.
	
	Conversely, assume that~\eqref{E:spoj-gop} is satisfied. Let $\sqa\in\RpZ$ be arbitrary. Define
	\begin{equation*}
		f=\sumZ a_n\chi_{(n-1,n]}.
	\end{equation*}
	Then we get
	\begin{align*}
		\left(
		\sum_{n\in\Z}\left(\sum_{i=-\infty}^n U(i,n) a_i\right)^q w_n
		\right)\sp{\frac 1q}
		&\le 2^{1+\frac1q} \left(
		\sum_{n\in\Z} \left(\int_{-\infty}^{n-\frac12}U(y,n)f(y)\dy\right)^q \int_{n-\frac12}^{n} w(t)\dt
		\right)\sp{\frac 1q} \\
		&\le 2^{1+\frac1q}\left(
		\sum_{n\in\Z}\int_{n-\frac12}^{n} \left(\int_{-\infty}^{t} U(y,t)f(y)\dy\right)^q w(t)\dt
		\right)\sp{\frac 1q}  \\
		&\le 2^{1+\frac1q}\left(
		\int_{-\infty}^{\infty} \left(\int_{-\infty}^{t}U(y,t)f(y)\dy\right)^q w(t)\dt
		\right)\sp{\frac 1q}.
	\end{align*}
	Now, using~\eqref{E:spoj-gop}, we obtain
	\begin{equation*}
		\left(
		\sum_{n\in\Z}\left(\sum_{i=-\infty}^n U(i,n) a_i\right)^q w_n
		\right)\sp{\frac 1q}
		\le 2^{1+\frac1q}C' \Bigg(\int_{-\infty}\sp{\infty}f(t)\sp pv(t)\dt\Bigg)\sp{\frac 1p} = 2^{1+\frac1q}C' \left(\sum_{n\in\Z}a_n^p v_n\right)^\jp.
	\end{equation*}
	Hence, \eqref{E:discrete-kernel-gop-dual-1} follows.
\end{proof}

\begin{proof}[Proof of Theorem~\ref{Th1.1}]
	The assertion follows from Theorem~\ref{T:main-prop1} and known necessary and sufficient conditions for~\eqref{E:spoj-gop} which were established in~\cite[Chap XI, $\S$1.5, Theorem 4]{KA} (assertions (i)--(vi)), \cite{Oin} (assertions (vii) and (viii)) and~\cite{Kre4} (assertions (ix) and (x)). It is easy to verify that since $U$ is a~regular kernel, its extension to $\R^{2}$ satisfies hypotheses of the cited results.
\end{proof}

Now we shall prove Theorem~\ref{T:kernel-main}. To this end we shall need a few auxiliary lemmas.

\begin{lemma}\label{ldisc}
	Let $\sqw\in\RpZ$ and $1 < D< \infty$. Then there exist elements $M\in\Z\cup\{\infty\}$ and $N\in\{-\infty\}\cup\Z$ and a strictly increasing sequence $\{n_k\}_{k= N-1}^{M}$ such that the following three conditions are satisfied:
	
	\textup{(i)} if $M\in\Z$, then $\sum_{i=n_M}^{\infty}  w_i>0$ and
	$\sum_{i=k}^{\infty} w_i=0$ for every $k> n_M$;
	if $N\in\Z$, then $n_{N-1}=-\infty$,
	
	\textup{(ii)} $\sum_{i=n_{k-1}+1}^{\infty} w_i\le D\sum_{i=n_k}^{\infty} w_i$ \quad \text{for every $k\in\Z\cap[N,M]$},
	
	\textup{(iii)} $D\sum_{i=n_k+1}^{\infty} w_i\leq \sum_{i=n_{k-1}}^{\infty}w_i$ \quad \text{for every $k\in\Z\cap(N,M)$}.
\end{lemma}

\begin{proof}
	Define the sets $A_k$ by
	\begin{equation*}
		A_k=\left\{ k\in \Z:D^{-k}< \sum_{i=k}^{\infty} w_i\le D^{-k+1}\right\}, \quad k\in \Z.
	\end{equation*}
	Let $\{A_{m_k}\}_{k=N}^{M}$ be the maximal subsequence of $\{A_k\}_{k\in \Z}$
	which contains only nonempty sets and let $n_k=\sup A_{m_k}$.
	It is clear that the sequence $\{n_k\}_{k=N}^M$ is strictly increasing.
	Moreover, if $M\in\Z$, then $\sum_{i= n_M}^{\infty} w_i>0$ and $\sum_{i=k}^{\infty}  w_i=0$ for every
	$k > n_M$. If $N\in\Z$, then $n_N>-\infty$ and we put $n_{N-1}=-\infty$. This shows (i). We have
	\begin{equation} \label{2.2}
		D^{-m_k}< \sum_{i=n_k}^{\infty}  w_i\le D^{-m_k+1} \quad \text{for every $k\in\Z\cap[N,M]$}.
	\end{equation}
	If $n_{k-1} < j\le n_{k}$, then $j\in A_{m_k}$. Together with \eqref{2.2}, this implies  that
	\begin{equation*}
		\sum_{i= n_{k-1}+1}^{\infty} w_i \le D \sum_{i=n_k}^{\infty}  w_i \quad \text{for every $k\in\Z\cap[N,M]$},
	\end{equation*}
	and (ii) follows. On the other hand,
	\begin{equation*}
		D\sum_{i= n_k+1}^{\infty}w_i \leq D^{-m_{k+1}+2}
		\leq D^{-m_{k-1}}\leq \sum_{i=n_{k-1}}^{\infty}  w_i\quad \text{for every $k\in\Z\cap(N,M)$},
	\end{equation*}
	and (iii) follows.
\end{proof}

\begin{lemma}\label{L:estimates-general}
	Let $ \sqw\in\RpZ$, $1 < D < \infty$ and let $\sqb\in\RpZ$ be nondecreasing. Let $M$, $N$ and $\{n_k\}_{k=N-1}^{M}$ be as in Lemma~\ref{ldisc}.
	Then
	\begin{equation*}
		\frac{D-1}{3D}\sum_{k=N}^{M}\Big ( \sum_{n=n_{k}}^{\infty} w_n \Big) b_{n_k}
		\le
		\sum_{n\in \Z} w_n  b_n
		\le
		D\sum_{k=N}^{M}\Big ( \sum_{n=n_{k}}^{\infty} w_n \Big) b_{n_k}.
	\end{equation*}
\end{lemma}

\begin{proof}
	We have
	\begin{align*}
		\sum_{n\in \Z} w_n  b_n
		& = \sum_{k=N}^{M} \sum_{n=n_{k-1}+1}^{n_{k}} w_nb_n
		\le D \sum_{k=N}^{M}\Big ( \sum_{n=n_{k}}^{\infty} w_n \Big) b_{n_k}.
	\end{align*}
	On the other hand,
	\begin{align*}
		\sum_{n\in \Z} w_n  b_n
		&\ge \frac 13  \sum_{k=N}^{M} \sum_{n=n_{k}}^{n_{k+2}} w_n b_n
		\ge \frac13 \sum_{k=N}^{M}\Big ( \sum_{n=n_{k}}^{n_{k+2}} w_n \Big) b_{n_k}
		\ge \frac13 (1-D^{-1}) \sum_{k=N}^{M}\Big ( \sum_{n=n_{k}}^{\infty} w_n \Big) b_{n_k}.
	\end{align*}
\end{proof}



\begin{lemma}\label{L:2.4}
	Let $\sqw\in\RpZ$, $0 < p \leq 1$ and $0 < q < \infty$. Suppose that $U^p$ is a regular kernel satisfying inequality \eqref{oin} with constant $C$. Assume further that
	\begin{equation*}
		2\max(1,2^{\qp-1})^2C^{\frac{q}{p}} \leq D < \infty
	\end{equation*}
	holds.  Assume that $M$, $N$ and $\{n_k\}_{k=N-1}^{M}$ are as in Lemma~\ref{ldisc}. Then
	\begin{align*}
		\sum_{n \in \Z}w_n \Big[ \sum_{ i=-\infty}^{n} U(i,n)^pa_i^p \Big]^\qp
		& \approx \sum_{k=N}^{M}\Big ( \sum_{n=n_{k}}^{\infty} w_n \Big) \Big[ \sum_{i=n_{k-1}+1}^{n_{k}} U(i,n_k)^pa_i^p \Big]^\qp
		\\
		& + \sum_{k=N+1}^{M}\Big ( \sum_{n=n_{k}}^{\infty}w_n \Big)U(n_{k-1},n_{k})^q \Big[ \sum_{i=-\infty}^{n_{k-1}} a_i^p \Big]^\qp\quad\text{for every $a\in\RpZ$}.
	\end{align*}	
\end{lemma}

\begin{proof}
	By Lemma~\ref{L:estimates-general}, it is clear that
	\begin{align*}
		\sum_{n\in \Z} w_n  \Big[ \sum_{i=-\infty}^{n} U(i,n)^pa_i^p \Big]^\qp \approx
		\sum_{k=N}^{M}\Big ( \sum_{n=n_{k}}^{\infty} w_n \Big) \Big[ \sum_{i=-\infty}^{n_k} U(i,n_k)^pa_i^p \Big]^\qp.
	\end{align*}
	Therefore, it is enough to show that
	\begin{align*}
		\sum_{k=N}^{M}\Big ( \sum_{n=n_{k}}^{\infty} w_n \Big) &\Big[ \sum_{i=-\infty}^{n_k} U(i,n_k)^pa_i^p \Big]^\qp
		\approx \sum_{k=N}^{M}\Big ( \sum_{n=n_{k}}^{\infty} w_n \Big) \Big[ \sum_{i=n_{k-1}+1}^{n_{k}} U(i,n_k)^pa_i^p \Big]^\qp
		\\
		& + \sum_{k=N+1}^{M}\Big ( \sum_{n=n_{k}}^{\infty} w_n \Big)U(n_{k-1},n_{k})^q \Big[ \sum_{i=-\infty}^{n_{k-1}} a_i^p \Big]^\qp\quad\text{for every $a\in\RpZ$}.
	\end{align*}	
	By the regularity of $U$, we have that
	\begin{align*}
		\sum_{k=N}^{M}D^{-m_{k}} \Big[ \sum_{i=-\infty}^{n_k} U(i,n_k)^pa_i^p \Big]^\qp
		&\le \max(1,2^{\qp-1}) \sum_{k=N}^{M} D^{-m_{k}}  \Big[ \sum_{i= n_{k-1}+1}^{n_k}  U(i,n_k)^pa_i^p \Big]^\qp
		\\
		&+ \max(1,2^{\qp-1})\sum_{k=N+1}^{M} D^{-m_{k}}  \Big[ \sum_{i=-\infty}^{n_{k-1}}U(i,n_k)^p a_i^p \Big]^\qp
		\\
		&\le \max(1,2^{\qp-1}) \sum_{k=N}^{M}D^{-m_{k}}  \Big[ \sum_{i=n_{k-1}+1}^{n_{k}} U(i,n_k)^pa_i^p \Big]^\qp
		\\
		&+ \max(1,2^{\qp-1})^2 C^{\frac qp}\sum_{k={N+1}}^{M} D^{-m_{k}} U(n_{k-1},n_{k})^q \Big[ \sum_{i=-\infty}^{n_{k-1}} a_i^p \Big]^\qp
		\\
		&+ \max(1,2^{\qp-1})^2 C^{\frac qp}\sum_{k=N+1}^{M} D^{-m_{k}}  \Big[ \sum_{i=-\infty}^{n_{k-1}} U(i,n_{k-1})^p a_i^p \Big]^\qp
		\\
		&\le \max(1,2^{\qp-1}) \sum_{k=N}^{M}D^{-m_{k}}  \Big[ \sum_{i=n_{k-1}+1}^{n_{k}} U(i,n_k)^pa_i^p \Big]^\qp
		\\
		&+ \max(1,2^{\qp-1})^2 C^{\frac qp}\sum_{k=N+1}^{M} D^{-m_{k}} U(n_{k-1},n_{k})^q \Big[ \sum_{i=-\infty}^{n_{k-1}} a_i^p \Big]^\qp
		\\
		&+ \frac12\sum_{k=N+1}^{M} D^{-m_{k-1}}  \Big[ \sum_{i=-\infty}^{n_{k-1}}U(i,n_{k-1})^p a_i^p \Big]^\qp.
	\end{align*}	
	Therefore,
	\begin{equation*}\begin{split}
			\sum_{k=N}^{M}D^{-m_{k}} \Big[ \sum_{i=-\infty}^{n_k} U(i,n_k)^pa_i^p \Big]^\qp
			& \le 2 \max(1,2^{\qp-1}) \sum_{k=N}^{M}D^{-m_{k}}  \Big[ \sum_{i=n_{k-1}+1}^{n_{k}} U(i,n_k)^pa_i^p \Big]^\qp
			\\
			& + 2\max(1,2^{\qp-1})^2 C^{\frac qp}\sum_{k=N+1}^{M} D^{-m_{k}} U(n_{k-1},n_{k})^q \Big[ \sum_{i=-\infty}^{n_{k-1}} a_i^p \Big]^\qp
		\end{split}
	\end{equation*}	
	holds. On the other hand, the estimate from below can be obtained easily by the monotonicity of $U$. The proof is complete.
\end{proof}

The following result is of a general nature. Note that there is no regularity or any other restriction required on the kernel.

\begin{theorem}\label{L:2.5}
	Let $\sqc\in\RpZ$, $0 < p, q < \infty$, and $U\colon\Z\times\Z\to[0,\infty)$. Assume that $M$, $N$ and $\{n_k\}_{k=N-1}^{M}$ are as in Lemma~\ref{ldisc}. Then
	\begin{align*}
		&\sup_{a\in\RpZ} \left(\sum_{k=N}^{M}c_k \Big[ \sum_{i=n_{k-1}+1}^{n_{k}} U(i,n_k)^pa_i^p \Big]^\qp\right)^{\frac1q}
		\Big ( \sum_{n\in \Z} a_n^p v_n \Big)^{-\frac1p}\\
		& \hskip+1cm \approx \sup_{a\in\RpZ} \left(\sum_{k=N}^{M}c_k \Big[ \sup_{n_{k-1}+1 \le i \le n_{k}} U(i,n_k)^pa_i^p \Big]^\qp\right)^{\frac1q}
		\Big ( \sum_{n\in \Z} a_n^p v_n \Big)^{-\frac1p}.
	\end{align*}	
\end{theorem}

\begin{proof} Fix $\sqa\in \RpZ$ such that $ \sum_{n\in \Z} a_n^p v_n <\infty$. Let $n_{k-1}+1 <  n_{k_0} < n_{k}$ be such that
	\[
	\sup_{n_{k-1}+1 \le i \le n_{k}} U(i,n_k)^p v_i^{-1} \le 2  U(n_{k_0},n_{k})^p v_{n_{k_0}}^{-1}.
	\]
	Now define the sequence $\sqb\in \RpZ$ by
	\[ b_n=\left\{
	\begin{array}{ll} \Big[ \sum_{i=n_{k-1}+1}^{n_{k}} a_i^p v_i \Big]^{\frac1p}v_{n_{k_0}}^{-\frac1p} & \textrm{if $n= n_{k_0}$ for $k\in\{N, \cdots, M\}$},\\
		0& \textrm{ otherwise}.
	\end{array}\right.
	\]
	We have
	$$
	\sumZ b_n^p v_n = \sum_{k=N}^{M} \sum_{j= n_{k-1}+1}^{ n_{k}} a_j^p v_j  \le \sumZ a_n^p v_n.
	$$
	Altogether, we obtain the following chain of relations,
	\begin{align*}
		&  \left(\sum_{k=N}^{M}c_k \Big[ \sum_{i=n_{k-1}+1}^{n_{k}} U(i,n_k)^pa_i^p \Big]^\qp\right)^{\frac1q}
		\Big ( \sum_{n\in \Z} a_n^p v_n \Big)^{-\frac1p} \\
		& \le   \left(\sum_{k=N}^{M}c_k \left[ \left( \sup_{n_{k-1}+1 \le i \le n_{k}} U(i,n_k)^p v_i^{-1} \right)  \sum_{i=n_{k-1}+1}^{ n_{k}}a_i^p v_i  \right]^\qp\right)^{\frac1q} \Big ( \sum_{n\in \Z} a_n^p v_n \Big)^{-\frac1p}\\
		& \le 2^{\frac{1}{p}}  \left(\sum_{k=N}^{M}c_k \left[ \left( U(n_{k_0},n_{k})^p v_{n_{k_0}}^{-1} \right) \sum_{i=n_{k-1}+1}^{n_{k}}a_i^p v_i  \right]^\qp\right)^{\frac1q} \Big ( \sum_{n\in \Z} a_n^p v_n \Big)^{-\frac1p}\\
		&\le 2^{\frac{1}{p}}  \left(\sum_{k=N}^{M}c_k\Big[  U(n_{k_0},n_k)^p b_{n_{k_0}}^{p}   \Big]^\qp\right)^{\frac1q} \Big ( \sum_{n\in \Z} b_n^p v_n \Big)^{-\frac1p}\\
		& \le 2^{\frac{1}{p}}  \left(\sum_{k=N}^{M}c_k \Big[ \sup_{n_{k-1}+1 \le i \le n_{k}} U(i,n_k)^p b_{i}^{p}   \Big]^\qp\right)^{\frac1q} \Big ( \sum_{n\in \Z} b_n^p v_n \Big)^{-\frac1p}.
	\end{align*}
	Taking supremum on both sides we obtain the upper estimate. On the other hand, the lower estimate follows easily by
	\begin{align*}
		\sup_{n_{k-1}+1 \le i \le n_{k}} U(i,n_k) a_{i}   \le  \sum_{i=n_{k-1}+1}^{n_{k}} U(i,n_k) a_{i}.
	\end{align*}
	Therefore, the statement follows.
\end{proof}

\begin{proof}[Proof of Theorem~\ref{T:kernel-main}]
	Since $0 < p \leq 1$, we have, for any $\{a_n\}\in\RpZ$,
	\begin{equation*}
		\sup_{-\infty<i\le n}U(i,n)a_i \le \sum_{i=-\infty}^{n} U(i,n)a_i\le \Bigg(\sum_{i=-\infty}^{n} U(i,n)^pa_i^p\Bigg)^\frac{1}{p}.
	\end{equation*}
	This establishes the implications $\eqref{E:discrete-kernel-strong}\Rightarrow\eqref{E:discrete-kernel-gop-dual-1}\Rightarrow \eqref{E:discrete-kernel-weak}$ with $C'\le C\le (C'')^{\frac 1p}$. Thus, we will be done if we show that~\eqref{E:discrete-kernel-weak}$\Rightarrow$\eqref{E:discrete-kernel-strong}. By Lemma~\ref{L:2.4}, Theorem~\ref{L:2.5} applied to $c_k=\sum_{n=n_{k}}^{\infty} w_n$ and~\cite[Lemma~3.4]{GOLP}, we have
	\begin{align*}
		\sup_{a\in\RpZ}&\left(\sum_{n\in\Z}w_n \Big[ \sum_{i=-\infty}^{n} U(i,n)^pa_i^p \Big]^\qp\right)^{\frac1q}\Big ( \sum_{n\in \Z} a_n^p v_n \Big)^{-\frac1p}\\
		& \approx
		\sup_{a\in\RpZ}\left(\sum_{k=N}^{M}\Big ( \sum_{n=n_{k}}^{\infty} w_n \Big) \Big[ \sum_{i=n_{k-1}+1}^{n_{k}} U(i,n_k)^pa_i^p \Big]^\qp\right)^{\frac1q}\Big ( \sum_{n\in \Z} a_n^p v_n \Big)^{-\frac1p}
		\\
		& + \sup_{a\in\RpZ}\left(\sum_{k=N+1}^{M}\Big ( \sum_{n=n_{k}}^{\infty} w_n \Big)U(n_{k-1},n_{k})^q \Big[ \sum_{i=-\infty}^{n_{k-1}} a_i^p \Big]^\qp\right)^{\frac1q}\Big ( \sum_{n\in \Z} a_n^p v_n \Big)^{-\frac1p}\\
		& \approx
		\sup_{a\in\RpZ}\left(\sum_{k=N}^{M}\Big ( \sum_{n=n_{k}}^{\infty} w_n \Big) \Big[ \sup_{n_{k-1}+1 \le i \le n_{k}} U(i,n_k)^pa_i^p \Big]^\qp\right)^{\frac1q}\Big ( \sum_{n\in \Z} a_n^p v_n \Big)^{-\frac1p}
		\\
		& + \sup_{a\in\RpZ}\left(\sum_{k=N+1}^{M}\Big ( \sum_{n=n_{k}}^{\infty} w_n \Big)U(n_{k-1},n_{k})^q \Big[ \sup_{-\infty< i\le n_{k-1}} a_i^p \Big]^\qp\right)^{\frac1q}\Big ( \sum_{n\in \Z} a_n^p v_n \Big)^{-\frac1p}\\
		& \leq
		\sup_{a\in\RpZ}\left(\sum_{k=N}^{M}\Big ( \sum_{n=n_{k}}^{\infty} w_n \Big) \Big[ \sup_{-\infty <i\le n_{k}} U(i,n_k)^pa_i^p \Big]^\qp\right)^{\frac1q}\Big ( \sum_{n\in \Z} a_n^p v_n \Big)^{-\frac1p}.
	\end{align*}	
	By Lemma~\ref{L:estimates-general}, we get
	\begin{align*}
		\sup_{a\in\RpZ}&\left(\sum_{n\in\Z}w_n \Big[ \sum_{i=-\infty}^{n} U(i,n)^pa_i^p \Big]^\qp\right)^{\frac1q}\Big ( \sum_{n\in \Z} a_n^p v_n \Big)^{-\frac1p}\\
		& \lesssim
		\sup_{a\in\RpZ}\left(\sum_{n\in\Z}w_n  \Big[ \sup_{-\infty< i \le n} U(i,n)^pa_i^p \Big]^\qp\right)^{\frac1q}\Big ( \sum_{n\in \Z} a_n^p v_n \Big)^{-\frac1p}.
	\end{align*}
	The proof is complete.
\end{proof}

\begin{proof}[Proof of Theorem~\ref{T:main-prop-dual-1}]
	By Theorem~\ref{T:kernel-main},~\eqref{E:discrete-kernel-gop-dual-1} holds if and only if there exists a constant $C''$ such that
	\begin{equation*}
		\Bigg(\sum_{n\in\Z}\Bigg(\sum_{i=-\infty}^{n} U(i,n)^{p}a_i\Bigg)^{\frac{q}{p}} w_n\Bigg)^{\frac pq}
		\le C''
		\sum_{n\in\Z}a_nv_n
	\end{equation*}
	holds for every sequence $\{a_n\} \in \RpZ$. Since $U$ is regular, so is $U^p$. Hence, using Theorem~\ref{T:main-prop1} with $p=1$, $q=\frac{q}{p}$ and $U=U^p$, establishes the claim.
\end{proof}

\section{Proofs of Theorems from Section~\ref{S:3} }

\begin{proof}[Proof of Theorem~\ref{T:six}]
	We first note that, trivially,
	\begin{equation}\label{E:elementary-inequalities-again}
		\sup_{-\infty<i\le n}U(i,n)a_i \le \sup_{-\infty<i\le n}\,U(i,n)\sup_{-\infty<j\le i} a_j \le \sup_{-\infty<i\le n}\,U(i,n)\sum_{j=-\infty}^{i} a_j.
	\end{equation}
	Moreover, monotonicity of $U$ gives
	\begin{equation} \label{E:new-elementary-inequality}
		\sup_{-\infty<i\le n}\,U(i,n)\sum_{j=-\infty}^{i} a_j\le \sum_{i=-\infty}^{n} U(i,n) a_i \le \Bigg(\sum_{i=-\infty}^{n} U(i,n)^pa_i^p\Bigg)^\frac{1}{p},
	\end{equation}
	and
	\begin{equation}\label{E:As}
		\sup_{-\infty<i\le n}\,U(i,n)\sum_{j=-\infty}^{i} a_j\le \sup_{-\infty<i\le n} U(i,n) \left(\sum_{j=-\infty}^{i} a_j^p\right)^{\frac{1}{p}}\le \Bigg(\sum_{i=-\infty}^{n} U(i,n)^pa_i^p\Bigg)^\frac{1}{p}.
	\end{equation}
	We thus obtain from~\eqref{E:elementary-inequalities-again} and~\eqref{E:new-elementary-inequality}
	\begin{equation*}
		B_1\le B_2\le B_3 \le B_4 \le B_6.
	\end{equation*}
	Next, it follows from~\eqref{E:As} that
	\begin{equation*}
		B_3\le B_5\le B_6.
	\end{equation*}
	Finally, by Theorem~\ref{T:kernel-main}, we get that $B_6\le C' B_1$. The proof is complete.
\end{proof}

\begin{proof}[Proof of Theorem~\ref{T:hux}]
	Note that, with a fixed $n\in\Z$, one has
	\begin{align*}
		\sup_{-\infty < i\le n}u_i\sup_{-\infty < j\le i} a_j
		& = \sup_{-\infty < i\le n} (\sup_{i\le j\le n}u_j)  a_i  \\
		\sup_{-\infty <i\le n }(\sup_{i\leq j\leq n} u_j)\sum_{j=-\infty}^{i}a_j &=
		\sup_{-\infty <i \le n}u_i\sum_{j=-\infty}^{i}a_j
		\\
		\sup_{-\infty <i \le n} u_i \left(\sum_{j=-\infty}^i a_j^{p}\right)^{\frac{1}{p}} &= \sup_{-\infty < i\le n}(\sup_{i\leq j\leq n} u_j)\left(\sum_{j=-\infty}^{i}a_j^{p}\right)^{\frac{1}{p}}.
	\end{align*}
	and the kernel $U(i,n)=\sup_{i\le j\le n}u_j$ is a regular kernel  so the result is a special case of Theorem~\ref{T:six}.
\end{proof}

\begin{proof}[Proof of Theorem~\ref{T:cd}]
	It follows immediately from Theorem~\ref{T:six}, namely from the equivalence of $B_3$ and $B_4$, that the inequality~\eqref{E:main-supremal-small} is equivalent to~\eqref{E:discrete-kernel-gop-dual-1}. Therefore, the assertion is a~consequence of Theorem~\ref{T:only-p-1}.
\end{proof}

\begin{proof}[Proof of Theorem~\ref{T:cont}]
	The assertion follows from comparing the corresponding conditions for each of the inequalities, which are well known in all possible cases. Moreover, this comparison shows that if $C$ and $C'$ are the optimal constants in the respective inequalities, then $C'\approx C$.
\end{proof}

The proof of Theorem~\ref{main} is rather involved and it requires some preliminary work.

\begin{lemma}\cite[Lemma~3.1]{GS} \label{Kernel estimate}
	Let $U\colon \R \times \R \rightarrow  (0,\infty) $ be a regular kernel. Then there exists an $\alpha_C$, $0 < \alpha_C < 1$, such that, for all
	$0 < \alpha \leq \alpha_C$ and any sequence $x_1 \leq x_2 \leq \dots \leq x_n$, one has
	\begin{equation*}
		U(x_1, x_n) \lesssim \left(\sum_{i=1}^{n-1} U(x_i, x_{i+1})^{\alpha} \right)^{\frac{1}{\alpha}}.
	\end{equation*}
\end{lemma}

\begin{definition}
	Let $\{x_k\}_{k\in \mathbb{Z}}$ be an increasing sequence in $(-\infty,\infty)$ such that
	\begin{equation*}
		\lim_{k\rightarrow -\infty} x_k= -\infty \quad \text{and} \quad \lim_{k\rightarrow \infty} x_k= \infty.
	\end{equation*}
	Then we say that  $\{x_k\}_{k\in \mathbb{Z}}$ is a~\textit{covering sequence}. We also admit increasing
	sequences  $\{x_k\}_{k={N-1}}^\infty$, $N\in \mathbb{Z}$ and $x_{N-1}:=- \infty$.
\end{definition}

\begin{definition}\cite{GP-discr}
	Let $\{x_k\}_{k\in \mathbb{Z}}$ be a
	sequence of positive real numbers. If $\sup_{k\in\mathbb{Z}} \frac{x_{k+1}}{x_k}  <1 $, then we say that
	$\{x_k\}_{k\in \mathbb{Z}}$ is \textit{strongly decreasing}.
\end{definition}
Throughout the rest of the paper we will assume that $\int_{x}^{\infty} w(t) dt <  \infty$ for every $x\in\R$. We then define
a strictly increasing sequence $\{x_k\}_{k= N-1}^{\infty}$ such that
\begin{equation*}
	\int_{x_k}^\infty w(t)dt = 2^{-k}, N\le k< \infty , \quad \text{and} \quad 2^{-N} < \int_{-\infty}^{\infty} w(t) dt \leq 2^{-N+1}.
\end{equation*}
Denote $x_{N-1}=- \infty$. Then it is clear that  $\{x_k\}_{k=N-1}^{\infty}$ is a covering sequence.

\begin{lemma}\label{Lemma 1}
	Let $0 < q < \infty$. Assume that $U$ is a regular kernel and $\{x_k\}_{k=N-1}^\infty$ is a covering sequence. Then
	\begin{align*}
		LHS\eqref{1-bis}
		& \approx \left( \sum_{k=N}^{\infty} 2^{-k} \left(\esssup_{x_{k-1} < y \leq
			x_{k}} U(y,x_k) \int_{x_{k-1}}^y f(t)\,dt \right)^q\right)^{\frac{1}{q}}\\
		&\hspace{0.5cm}+ \left( \sum_{k=N}^{\infty} 2^{-k} U(x_{k-1},x_k)^q
		\left(\int_{-\infty}^{x_{k-1}} f(t)\,dt\right)^q \right)^{\frac{1}{q}}
	\end{align*}
	holds for all $f\in\mathcal{M}_+$ with constants of equivalence depending only on $q$ and the constant of regularity of $U$. In particular, the constants of equivalence are independent of the covering sequence.
\end{lemma}

\begin{proof}
	It is easy to see that,
	\begin{align}
		LHS\eqref{1-bis} & = \left(\sum_{k=N}^{\infty} \int_{x_{k-1}}^{x_k} \left(\esssup_{-\infty < y\leq x}U(y,x)\int_{-\infty}^y f(t)\,dt\right)^q w(x)dx\right)^{\frac{1}{q}} \notag  \\
		& \approx \left( \sum_{k=N}^{\infty}  2^{-k} \left(\esssup_{-\infty < y\le x_k} U(y,x_k)
		\int_{-\infty}^y f(t)\,dt\right)^q \right)^{\frac{1}{q}} \label{LHS.3.4 equiv.}.
	\end{align}
	It follows from~\cite[Lemma~3.2(ii)]{GP-discr} that if $\{\alpha_k\}_{k\in\Z}$ and $\{\beta_k\}_{k\in\Z}$ are sequences of nonnegative real numbers such that $\{\alpha_k\}$ is strongly decreasing, then
	\begin{align*}
		\sum_{k \in \mathbb{Z}} \alpha_k \sup_{-\infty < i\leq k} \beta_i   &\approx \sum_{k \in \mathbb{Z}} \alpha_k \beta_k.
	\end{align*}
	Applying this result and using also the regularity of $U$, we get
	\begin{align}\label{LHS1}
		LHS\eqref{1-bis} & \approx  \left( \sum_{k=N}^{\infty}  2^{-k}\sup_{N\le i\leq k} \left(\esssup_{x_{i-1}< y\le x_i} U(y,x_k) 	\int_{-\infty}^y f(t)\,dt\right)^q \right)^{\frac{1}{q}}\notag\\
		& \lesssim  \left( \sum_{k=N}^{\infty}  2^{-k} \sup_{N\le i\leq k} \left(\esssup_{x_{i-1}< y\le x_i} U(y,x_i)
		\int_{-\infty}^y f(t)\,dt\right)^q \right)^{\frac{1}{q}}\notag\\
		& \hspace{1.5cm} +   \left( \sum_{k=N}^{\infty}  2^{-k}	\sup_{N\le i \leq k} \left( U(x_i,x_k) 	             \int_{-\infty}^{x_i} f(t)\,dt\right)^q \right)^{\frac{1}{q}}\notag\\
		&  \approx \left( \sum_{k=N}^{\infty}  2^{-k}  \left(\esssup_{x_{k-1}< y\le x_k} U(y,x_k)
		\int_{-\infty}^y f(t)\,dt\right)^q \right)^{\frac{1}{q}}\notag\\
		& \hspace{1.5cm} + \left( \sum_{k=N}^{\infty}  2^{-k} \sup_{N\le i \leq k} \left( U(x_i,x_k)                  \int_{-\infty}^{x_i} f(t)\,dt\right)^q \right)^{\frac{1}{q}}\notag\\
		& \approx \left( \sum_{k=N}^{\infty}  2^{-k}  \left(\esssup_{x_{k-1}< y\le x_k} U(y,x_k)\int_{x_{k-1}}^y         f(t)\,dt\right)^q \right)^{\frac{1}{q}}\notag\\
		&\hspace{1.5cm} +\left( \sum_{k=N}^{\infty}  2^{-k}	 \left(U(x_{k-1},x_k)\int_{-\infty}^{x_{k-1}} f(t)\,dt \right)^q \right)^{\frac{1}{q}}\notag\\
		& \hspace{2cm} + \left( \sum_{k=N}^{\infty}  2^{-k} \sup_{N\le  i \leq k} \left( U(x_i,x_k) \int_{-\infty}^{x_i} f(t)\,dt\right)^q \right)^{\frac{1}{q}}.
	\end{align}
	In view of Lemma~\ref{Kernel estimate}, there exists  an $\alpha_C$: $0 < \alpha_C < 1$ such that, for all $0 < \alpha \leq \alpha_C$,
	\begin{equation}\label{kernel est.}
		U(x_i, x_k) \lesssim \left(\sum_{j=i}^{k-1} U(x_{j}, x_{j+1})^{\alpha} \right)^{\frac{1}{\alpha}}
	\end{equation}
	holds.
	It follows from~\cite[Lemma~3.1(ii)]{GP-discr} that if $\{\alpha_k\}_{k\in\Z}$ and $\{\beta_k\}_{k\in\Z}$ are sequences of nonnegative real numbers such that $\{\alpha_k\}$ is strongly decreasing and $p\in(0,\infty)$, then
	\begin{align}\label{E:GP-lemma2}
		\sum_{k \in \mathbb{Z}} \alpha_k^p \bigg(\sum_{i=-\infty}^{k} \beta_i \bigg)^p  &\approx \sum_{k \in
			\mathbb{Z}} \alpha_k^p  \beta_k^p.
	\end{align}
	Applying this result, we get
	\begin{align*}
		&\left( \sum_{k=N}^{\infty}  2^{-k} 	\sup_{N\le i\leq k} \left( U(x_i,x_k) \int_{-\infty}^{x_i} f(t)\,dt\right)^q \right)^{\frac{1}{q}} \\
		& \hspace{2cm}\lesssim \left( \sum_{k=N}^{\infty}  2^{-k} \sup_{N\le  i \leq k}  \left(\sum_{j=i}^{k-1} 	U(x_{j}, x_{j+1})^{\alpha} \right)^{\frac{q}{\alpha}} \left(\int_{-\infty}^{x_i} f(t)\,dt\right)^q 	\right)^{\frac{1}{q}}\\
		& \hspace{2cm} \leq \left( \sum_{k=N}^{\infty}  2^{-k} \left(\sum_{j=N}^{k-1} U(x_{j}, x_{j+1})^{\alpha}	\left(\int_{-\infty}^{x_{j}} f(t)\,dt\right)^\alpha \right)^{\frac{q}{\alpha}} \right)^{\frac{1}{q}}\\
		& \hspace{2cm} \approx  \left( \sum_{k=N}^{\infty}  2^{-k} U(x_{k-1}, x_k)^q \left(\int_{-\infty}^{x_{k-1}} f(t)\,dt\right)^q \right)^{\frac{1}{q}}.
	\end{align*}
	Plugging this in \eqref{LHS1}, the upper estimate follows. On the other hand, the lower estimate can be obtained easily from \eqref{LHS.3.4 equiv.} using the monotonicity of $U$, and the statement follows.
\end{proof}

\begin{lemma}\label{Lemma 2}
	Let $1 \leq p < \infty$ and $0 < q < \infty$. Assume that $U$ is a regular kernel and $\{x_k\}_{k=N-1}^\infty$ is a covering sequence. Then
	\begin{align*}
		LHS\eqref{2-bis}
		& \approx \left( \sum_{k=N}^{\infty} 2^{-k} \left(\int_{x_{k-1}}^{x_{k}} U(y,x_k)^p f(y)\dy \right)^{\frac{q}{p}}\right)^{\frac{p}{q}}\\
		&\hspace{0.5cm}+ \left( \sum_{k=N}^{\infty} 2^{-k} U(x_{k-1},x_k)^q
		\left(\int_{-\infty}^{x_{k-1}} f(t)\,dt\right)^{\frac{q}{p}} \right)^{\frac{p}{q}}
	\end{align*}
	holds for all $f\in\mathcal{M}_+$ with constants of equivalence depending only on $q$ and the constant of regularity of $U$. In particular, the constants of equivalence are independent of the covering sequence.
\end{lemma}

\begin{proof}
	As in \eqref{LHS.3.4 equiv.}, it is clear that
	\begin{equation} \label{LHS.3.5 equiv.}
		LHS\eqref{2-bis} \approx \left(\sum_{k=N}^{\infty}2^{-k} \left(\int_{-\infty}^{x_k} U(y,x_k)^p f(y) \dy \right)^{\frac{q}{p}} \right)^{\frac{p}{q}}.
	\end{equation}
	Using the regularity of $U$ and applying~\eqref{E:GP-lemma2}, we have that
	\begin{align}
		LHS\eqref{2-bis} &\approx \left(\sum_{k=N}^{\infty}2^{-k} \left(\sum_{i=N}^k\int_{x_{i-1}}^{x_i} U(y,x_k)^p f(y) \dy \right)^{\frac{q}{p}} \right)^{\frac{p}{q}}\notag\\
		& \lesssim \left(\sum_{k=N}^{\infty}2^{-k} \left(\sum_{i=N}^k \int_{x_{i-1}}^{x_{i}}
		U(y,x_i)^p f(y) dy \right)^{\frac{q}{p}} \right)^{\frac{p}{q}} \notag\\
		& \hspace{0.5cm} + \left(\sum_{k=N}^{\infty}2^{-k} \left(\sum_{i=N}^{k}  U(x_i,x_k)^p\int_{x_{i-1}}^{x_{i}}  f(y) dy \right)^{\frac{q}{p}} \right)^{\frac{p}{q}} \notag\\
		&\lesssim	\left(\sum_{k=N}^{\infty}2^{-k} \left( \int_{x_{k-1}}^{x_{k}}
		U(y,x_k)^p f(y) dy \right)^{\frac{q}{p}} \right)^{\frac{p}{q}} \notag\\
		& \hspace{0.5cm} +\left(\sum_{k=N}^{\infty}2^{-k} \left(\sum_{i=N}^{k-1} U(x_i,x_k)^p \int_{x_{i-1}}^{x_{i}}  f(y) dy \right)^{\frac{q}{p}} \right)^{\frac{p}{q}}. 	 \label{LHS2}
	\end{align}
	Using \eqref{kernel est.}, and applying Minkowski's inequality with $\frac{p}{\alpha} >1$, we get
	\begin{align*}
		& \left(\sum_{k=N}^{\infty}2^{-k} \left(\sum_{i=N}^{k-1} U(x_i,x_k)^p \int_{x_{i-1}}^{x_{i}}  f(y) dy \right)^{\frac{q}{p}} \right)^{\frac{p}{q}} \\
		& \hspace{2cm}\lesssim \left(\sum_{k=N}^{\infty} 2^{-k} \left(\sum_{i=N}^{k-1} 	\left( \sum_{j=i}^{k-1} U(x_{j}, x_{j+1})^{\alpha} \right)^{\frac{p}{\alpha}} \int_{x_{i-1}}^{x_{i}}  f(y) dy \right)^{\frac{q}{p}} \right)^{\frac{p}{q}} \\
		&\hspace{2cm} \lesssim \left(\sum_{k=N}^{\infty} 2^{-k} \left(\sum_{j=N}^{k-1} U(x_{j},x_{j+1})^{\alpha} 	\left(\sum_{i=N}^{j}\int_{x_{i-1}}^{x_{i}}  f(y) dy \right)^{\frac{p}{\alpha}} 	\right)^{\frac{q}{\alpha}} \right)^{\frac{p}{q}}\\
		&\hspace{2cm} = \left(\sum_{k=N}^{\infty} 2^{-k} \left(\sum_{j=N}^{k-1} U(x_{j},x_{j+1})^{\alpha} \left(\int_{-\infty}^{x_{j}}  f(y) dy \right)^{\frac{p}{\alpha}}
		\right)^{\frac{q}{\alpha}} \right)^{\frac{p}{q}}.
	\end{align*}
	Therefore,~\eqref{E:GP-lemma2} yields
	\begin{equation*}
		\left(\sum_{k=N}^{\infty}2^{-k} \left(\sum_{i=N}^{k-1} U(x_i,x_k)^p \int_{x_{i-1}}^{x_{i}}  f(y) dy \right)^{\frac{q}{p}} \right)^{\frac{p}{q}} \lesssim \left(\sum_{k=N}^{\infty}2^{-k} U(x_{k-1}, x_k)^q \left(\int_{-\infty}^{x_{k-1}} f(y) dy 	\right)^{\frac{q}{p}} \right)^{\frac{p}{q}}.
	\end{equation*}
	Inserting this estimate in \eqref{LHS2} completes the proof of the upper estimate. The estimate from below can be obtained easily from \eqref{LHS.3.5 equiv.} using the monotonicity of $U$. The proof is complete.
\end{proof}

\begin{lemma} \label{Lemma 3}
	Let $1\le p<\infty$ and $0<q< \infty$. Assume that $U$ is a regular kernel and let $\{x_k\}_{k=N-1}^{\infty}$ is a covering sequence. Then
	\begin{align*}
		LHS\eqref{3-bis} & \approx \left( \sum_{k=N}^{\infty} 2^{-k} \left(\esssup_{x_{k-1} < y \leq x_{k}} U(y,x_k)^p \int_{x_{k-1}}^y f (t)\,dt\right)^{\frac{q}{p}}\right)^{\frac{p}{q}}\\
		&\hspace{0.5cm}+ \left( \sum_{k=N}^{\infty} 2^{-k} U(x_{k-1},x_k)^q \left(\int_{-\infty}^{x_{k-1}} f(t)\,dt\right)^{\frac{q}{p}} \right)^{\frac{p}{q}}
	\end{align*}
	holds for all $f\in\mathcal{M}_+$ with constants of equivalence depending only on $q$ and the constant of regularity of $U$. In particular, the constants of equivalence are independent of the covering sequence.
\end{lemma}

\begin{proof}
	The proof is similar to that of Lemma~\ref{Lemma 1}, therefore we omit it.
\end{proof}

\begin{proof}[Proof of Theorem~\ref{main}]
	
	We will show $\eqref{1-bis}\Rightarrow\eqref{2-bis}\Rightarrow \eqref{3-bis}\Rightarrow\eqref{1-bis}$.
	
	{\rm  \eqref{1-bis} $\Rightarrow$\eqref{2-bis}.}
	Assume first that \eqref{1-bis} holds. Then, in view of Lemma~\ref{Lemma 1},
	\begin{align}
		\label{In 2.1}
		\left( \sum_{k=N}^{\infty} 2^{-k} \left(\esssup_{ x_{k-1} < y  \leq x_k} U(y,x_k) \int_{x_{k-1}}^y f(t)\,dt
		\right)^{q}\right)^{\frac{1}{q}}
		&\lesssim C_1\left( \sum_{k=N}^{\infty}  \int_{x_{k-1}}^{x_{k}}
		f(t)^p v(t)\,dt\right)^{\frac{1}{p}}\\
		\intertext{and}
		\label{In 2.2}
		\left( \sum_{k=N}^{\infty}  2^{-k} U(x_{k-1},x_k)^q \left(\int_{-\infty}^{x_{k-1}} f(t)\,dt\right)^{q} \right)^{\frac{1}{q}}
		&\lesssim C_1\left(\sum_{k=N}^{\infty}  \int_{x_{k-1}}^{x_{k}}
		f(t)^p v(t)\,dt\right)^{\frac1p}
	\end{align}
	hold for all $f\in\mathcal{M}_+$. Let $h_k \in \mathcal{M}_+$ be a function which saturates the Hardy inequalities (see \cite{OK}), that is, function satisfying
	\begin{align*}
		\supp h_k \subset [x_{k-1}, x_{k}], \quad \int_{x_{k-1}}^{x_{k}} h_k(t)^pv(t)\,dt =1\quad \\
		\intertext{and}
		\esssup_{x_{k-1} < y \leq x_{k}} U(y, x_k) \int_{x_{k-1}}^y h_k(t)\,dt \gtrsim \esssup_{ x_{k-1} < y \leq x_k} U(y,x_k) \sigma_p(x_{k-1},y),
	\end{align*}
	in which the constant of $\gtrsim$ depends only on $p$ and the constant of regularity of $U$.
	Then we define $f = \sum_{m=N}^{\infty} a_m h_m$, where  $\{a_k\} \in \RpZ$. Testing inequality \eqref{In 2.1}, we have
	\begin{equation} \label{In 2.21}
		\left( \sum_{k=N}^{\infty}  2^{-k} \left(\esssup_{x_{k-1} < y \leq x_{k}} U(y,x_k) \sigma_p(x_{k-1},y)\right)^q a_k^{q} \right)^{\frac{1}{q}} \lesssim  C_1 \left(\sum_{k=N}^{\infty}	a_k^p\right)^{\frac1p}.
	\end{equation}
	Using the well-known saturation of the H\"older inequality, we find the functions $g_k \in \mathcal{M}_+$, $k\in\Z$, such that
	\begin{equation*}
		\supp g_k \subset [x_{k-1}, x_{k}], \quad \int_{x_{k-1}}^{x_{k}} g_k(t)^pv(t)\,dt =1\quad \text{and}
		\int_{x_{k-1}}^{x_{k}}  g_k(t)\,dt \ge\frac{1}{2} \sigma_p(x_{k-1},x_k).
	\end{equation*}
	Testing inequality \eqref{In 2.2} with $f = \sum_{m=N}^{\infty}  a_m g_m$  where  $\{a_k\} \in \RpZ$, we get	
	\begin{equation} \label{In 2.22}
		\left( \sum_{k=N}^{\infty}  2^{-k} U(x_{k-1},x_k)^q \left(\sum_{i=N}^{k-1} a_i \sigma_p(x_{i-1},x_i) \right)^{q} \right)^{\frac{1}{q}} \lesssim  C_1 \left(\sum_{k=N}^{\infty}	a_k^p\right)^{\frac{1}{p}}.
	\end{equation}
	By Theorem~\ref{disc. equiv.}, the inequality \eqref{In 2.22} is equivalent to the following one:
	\begin{equation} \label{In 2.23}
		\left( \sum_{k=N}^{\infty}  2^{-k} U(x_{k-1},x_k)^q \left(\sum_{i=N}^{k-1} a_i^p \sigma_p(-\infty,x_i)^p \right)^{\frac{q}{p}} \right)^{\frac{1}{q}}  \lesssim  C_1 \left(\sum_{k=N}^{\infty}   a_k^p\right)^{\frac{1}{p}}.
	\end{equation}
	On the other hand, by Lemma~\ref{Lemma 2} and  using  H\"{o}lder inequality, we get
	\begin{align*}
		LHS\eqref{2-bis}
		&\lesssim \left( \sum_{k=N}^{\infty} 2^{-k} \left(\esssup_{x_{k-1} \leq y < x_{k}} U(y,x_k) \sigma_p(-\infty,y)\right)^q\left(\int_{x_{k-1}}^{x_k} f(y)\sigma_p(-\infty,y)^{-p}dy \right)^{\frac{q}{p}}\right)^{\frac{p}{q}} \notag\\
		+& \left( \sum_{k=N}^{\infty} 2^{-k} U(x_{k-1},x_k)^q
		\left(\sum_{i=N}^{k-1} \left(\int_{x_{i-1}}^{x_i} f(y)\sigma_p(-\infty,y)^{-p}dy\right)\sigma_p(-\infty,x_i)^p\right)^{\frac{q}{p}} \right)^{\frac{p}{q}}
		\\
		\lesssim & \left( \sum_{k=N}^{\infty} 2^{-k} \left(\esssup_{x_{k-1} \leq y < x_{k}} U(y,x_k) \sigma_p(x_{k-1},y)\right)^q\left(\int_{x_{k-1}}^{x_k} f(y)\sigma_p(-\infty,y)^{-p} dy\right)^{\frac{q}{p}}\right)^{\frac{p}{q}} \notag\\
		+& \left( \sum_{k=N}^{\infty} 2^{-k} U(x_{k-1},x_k)^q
		\sigma_p(-\infty,x_{k-1})^q \left(\int_{x_{k-1}}^{x_k} f(y)\sigma_p(-\infty,y)^{-p}dy \right)^{\frac{q}{p}} \right)^{\frac{p}{q}}
		\\
		+& \left( \sum_{k=N}^{\infty} 2^{-k} U(x_{k-1},x_k)^q
		\left(\sum_{i=N}^{k-1} \left(\int_{x_{i-1}}^{x_i} f(y)\sigma_p(-\infty,y)^{-p} dy\right)\sigma_p(-\infty,x_i)^p\right)^{\frac{q}{p}} \right)^{\frac{p}{q}}
	\end{align*}
	
	Using inequalities \eqref{In 2.21} and  \eqref{In 2.23}  for sequences $a_k=\left(\int_{x_{k-1}}^{x_k} f(y) \sigma_p(-\infty,y)^{-p}\,dy \right)^{\frac1p}$, also inequality  \eqref{In 2.23} for sequences $a_k= \left( \int_{x_{k}}^{x_{k+1}} f(y) \sigma_p(-\infty,y)^{-p}\,dy\right)^{\frac1p}$, we obtain \eqref{2-bis} with $C_2 \lesssim C_1^p$.
	
	{\rm  $\eqref{2-bis}\Rightarrow\eqref{3-bis}$.}
	By Lemma~\ref{Lemma 2} and Lemma~\ref{Lemma 3} and estimate
	
	\begin{equation*}
		\esssup_{x_{k-1}< y \leq x_{k}} U(y,x_k)^p
		\int_{x_{k-1}}^{y} f(y)\,dy \le  \int_{x_{k-1}}^{x_{k}} U(y,x_k)^p f(y)dy.
	\end{equation*}
	we have
	\begin{equation*}
		LHS\eqref{3-bis}\lesssim LHS\eqref{2-bis}.
	\end{equation*}
	and  implication $\eqref{2-bis}\Rightarrow\eqref{3-bis}$ follows with $C_3 \leq C_2$.
	
	\

	{\rm \eqref{3-bis} $\Rightarrow$\eqref{1-bis}.}
	Assume first that \eqref{3-bis} holds. Then, in view of Lemma~\ref{Lemma 3},
	\begin{equation}\label{In 1.1}
		\bigg( \sum_{k=N}^{\infty} 2^{-k} \bigg(\esssup_{ x_{k-1} < y  \leq x_k} U(y,x_k)^p \int_{x_{k-1}}^y f(t)\,dt \bigg)^{\frac{q}{p}}\bigg)^{\frac{p}{q}} \lesssim  C_3 \sum_{k=N}^{\infty}  \int_{x_{k-1}}^{x_{k}} f(y) \sigma_p(-\infty,y)^{-p} dy
	\end{equation}
	and
	\begin{equation}\label{In 1.2}
		\bigg( \sum_{k=N}^{\infty}  2^{-k} U(x_{k-1},x_k)^q \bigg(\int_{-\infty}^{x_{k-1}} f(t)\,dt\bigg)^{\frac{q}{p}}	\bigg)^{\frac{p}{q}}  \lesssim  C_3 \sum_{k=N}^{\infty}  \int_{x_{k-1}}^{x_{k}} f(y) \sigma_p(-\infty,y)^{-p} dy,
	\end{equation}
	hold for all $f\in\mathcal{M}_+$.
	
	Let $h_k \in	\mathcal{M}_+$ be a function that saturates the Hardy inequalities (see \cite{OK}), that is, a function satisfying
	\begin{align*}
		\supp h_k \subset [x_{k-1}, x_{k}], \quad \int_{x_{k-1}}^{x_{k}} h_k(y) \sigma_p(-\infty,y)^{-p} dy=1\quad \\
		\intertext{and}
		\esssup_{x_{k-1} < y \leq x_{k}} U(y, x_k)^p \int_{x_{k-1}}^y h_k(t)\,dt \gtrsim \esssup_{ x_{k-1} < y \leq x_k} U(y,x_k)^p \sigma_p(-\infty,y)^p,
	\end{align*}
	in which the constant of $\gtrsim$ depends only on $p$ and the constant of regularity of $U$.
	Then we define $f = \sum_{m=N}^{\infty} a_m^p h_m$, where  $\{a_k\} \in \RpZ$.  Testing inequality \eqref{In 1.1}, we have
	\begin{equation} \label{In 1.21}
		\left( \sum_{k=N}^{\infty}  2^{-k} \left(\esssup_{x_{k-1} < y \leq x_{k}} U(y,x_k) \sigma_p(-\infty,y)\right)^q a_k^{q}
		\right)^{\frac{1}{q}}  \lesssim C_3^{\frac{1}{p}} \left(\sum_{k=N}^{\infty}	a_k^p\right)^{\frac1p}.
	\end{equation}
	We can find  functions $g_k \in \mathcal{M}_+$, $k\in\Z$, such that
	\begin{equation*} \supp g_k \subset [x_{k-1}, x_{k}], \quad \int_{x_{k-1}}^{x_{k}} g_k(y) \sigma_p(-\infty,y)^{-p}dy =1\quad \text{and} 	 \int_{x_{k-1}}^{x_{k}}  g_k(t)\,dt \ge\frac{1}{2} \sigma_p(-\infty,x_k)^{p}.
	\end{equation*}
	Testing inequality \eqref{In 1.2} with $f = \sum_{m=N}^{\infty}  a_m g_m$,  where  $\{a_k\} \in \RpZ$, we have
	
	\begin{equation} \label{In 1.22}
		\left( \sum_{k=N}^{\infty}  2^{-k} U(x_{k-1},x_k)^q \left(\sum_{i=N}^{k-1} a_i \sigma_p(-\infty,x_i)^p \right)^{\frac{q}{p}}
		\right)^{\frac{p}{q}}  \lesssim C_3 \sum_{k=N}^{\infty}   a_k.
	\end{equation}
	By Theorem~\ref{disc. equiv.} the inequality \eqref{In 1.22} is equivalent to the following one:
	\begin{equation} \label{In 1.23}
		\left( \sum_{k=N}^{\infty}  2^{-k} U(x_{k-1},x_k)^q \left(\sum_{i=N}^{k-1} a_i \sigma_p(x_{i-1},x_i) \right)^{q} \right)^{\frac{1}{q}} \lesssim C_3^{\frac{1}{p}} \left(\sum_{k=N}^{\infty}   a_k^p\right)^{\frac{1}{p}}.
	\end{equation}
	By Lemma~\ref{Lemma 1},  using the Hardy inequality  (see \cite{OK}) and the H\"{o}lder inequality, we get
	\begin{align*}
		LHS\eqref{1-bis}& \lesssim \left( \sum_{k=N}^{\infty} 2^{-k} \left(\esssup_{x_{k-1} < y \leq
			x_{k}} U(y,x_k) \sigma_p(x_{k-1},y)\right)^q\left(\int_{x_{k-1}}^{x_k} f(t)^pv(t)\,dt \right)^{\frac{q}{p}}\right)^{\frac{1}{q}}\\
		&\hspace{0.5cm}+ \left( \sum_{k=N}^{\infty} 2^{-k} U(x_{k-1},x_k)^q
		\left(\sum_{i=-\infty}^{k-1} \left(\int_{x_{i-1}}^{x_i} f(t)^pv(t)\,dt\right)^{\frac1p}\sigma_p(x_{i-1},x_i)\right)^q \right)^{\frac{1}{q}}.
	\end{align*}
	Using inequalities \eqref{In 1.21} and \eqref{In 1.23} for $a_k= \left(\int_{x_{k-1}}^{x_k} f^pv\right)^{\frac1p}$ we obtain \eqref{1-bis}  with $C_1 \lesssim C_3^{\frac{1}{p}}$.
\end{proof}

\begin{proof}[Proof of Theorem~\ref{T:supremalpge}]
	The assertion follows from the combination of Theorem~\ref{T:main-prop1}, Theorem~\ref{T:main-supremal} and Theorem~\ref{main}.
\end{proof}

\begin{proof}[Proof of Theorem~\ref{T:supremalpge1}]
	It follows immediately from Theorem~\ref{T:supremalpge}, namely from the equivalence of $C$ and $C'$, that the inequality~\eqref{E:main-supremal-small} is equivalent to~\eqref{E:C'}. Therefore, the assertion is a~consequence of Theorem~\ref{Th1.1}.
\end{proof}


\

\end{document}